\documentclass[12pt]{amsart}
\usepackage{geometry}
\usepackage{amsmath, amssymb}
\usepackage{amscd}
\usepackage{enumitem}
\usepackage{comment}
\usepackage[english]{babel}
\usepackage{graphicx}
\numberwithin{equation}{section}

\newtheorem{thm}{Theorem}[section]
\newtheorem{lem}[thm]{Lemma}
\newtheorem{prop}[thm]{Proposition}
\newtheorem{cor}[thm]{Corollary}

\newtheorem*{mainquestion}{The Main Question}
\newtheorem*{thma}{Theorem A}
\newtheorem*{thmb}{Theorem B}

\theoremstyle{definition}

\theoremstyle{remark}

\newtheorem{ex}[thm]{Example}

\newcommand{\ebf}[1]{\textbf{\emph{#1}}}

\newcommand{\comdia}[8]
{
\begin{displaymath}
\begin{CD}
#1 @> #5 >> #2\\
@VV #7 V @VV #8 V\\
#3 @> #6 >> #4
\end{CD}
\end{displaymath}
}

\begin{document}

\title{Mating the Basilica with a Siegel Disk}
\author{Jonguk Yang}
\date{\today}

\maketitle

\begin{abstract}
Let $f_\mathbf{S}$ be a quadratic polynomial with a fixed Siegel disc of bounded type. Using an adaptation of complex a priori bounds for critical circle maps, we prove that $f_\mathbf{S}$ is conformally mateable with the basilica polynomial $f_\mathbf{B}(z) := z^2-1$.
\end{abstract}

\section{The Definition of Mating} \label{defnmating}

The simplest non-linear examples of holomorphic dynamical systems are given by quadratic polynomials in $\mathbb{C}$. By a linear change of coordinates, any quadratic polynomial can be uniquely normalized as
\begin{displaymath}
f_c (z) := z^2 + c, \hspace{5 mm} c \in \mathbb{C}.
\end{displaymath}
This is referred to as the \ebf{quadratic family}.

The critical points for $f_c$ are $\infty$ and $0$. Observe that $\infty$ is a superattracting fixed point for $f_c$. Let $\mathbf{A}^\infty_c$ be the attracting basin of $\infty$. It follows from the maximum modulus principle that $\mathbf{A}^\infty_c$ is a connected set. The complement of $\mathbf{A}^\infty_c$ is called the filled Julia set $K_c$. The boundary of $K_c$ is equal to the Julia set $J_c$.

The non-escape locus in the parameter space for $f_c$ (referred to as the \ebf{Mandelbrot set}) is defined as a compact subset of $\mathbb{C}$:
\begin{displaymath}
\mathcal{M} := \{c \in \mathbb{C} \hspace{2mm} | \hspace{2mm} 0 \notin \mathbf{A}^\infty_c\},
\end{displaymath}
which is known to be connected (see \cite{DH}). It is not difficult to prove that $J_c$ is connected (and therefore, $\mathbf{A}^\infty_c$ is simply connected) if and only if $c \in \mathcal{M}$. In fact, if $c \notin \mathcal{M}$, then $J_c=K_c$ is a Cantor set, and the dynamics of $f_c$ restricted to $J_c$ is conjugate to the dyadic shift map (see \cite{M2}). We also define the following subset of the Mandelbrot set:
\begin{displaymath}
\mathcal{L} := \{c \in \mathcal{M} \hspace{2mm} | \hspace{2mm} J_c \text{ is locally connected}\}.
\end{displaymath}

\begin{figure}[h]
\centering
\includegraphics[scale=0.6]{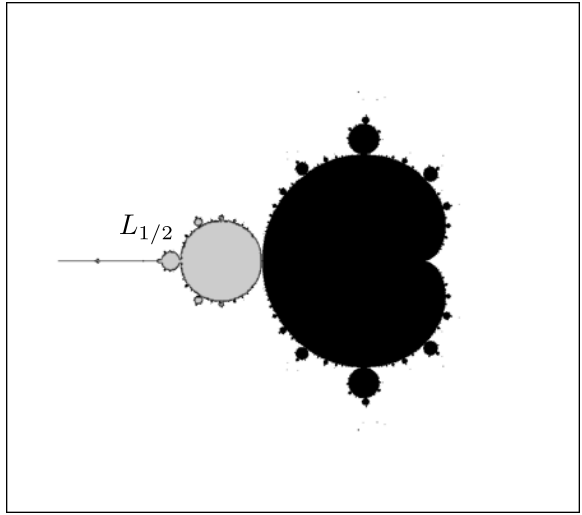}
\caption{The Mandelbrot set $\mathcal{M}$. The 1/2-limb $L_{1/2}$ is highlighted.}
\label{fig:mandelbrot}
\end{figure}

The quadratic family has been the center of attention in holomorphic dynamics for the past three decades, and we now have an almost complete understanding of its dynamics (see e.g. \cite{M2}). It should be noted however, that the main conjecture in the field (the local connectedness property of the Mandelbrot set, or \emph{MLC} for short) remains open (see \cite{DH}).

In contrast to the quadratic family, the dynamics of non-polynomial quadratic rational maps is still a wide open area of research. In this section, we describe a construction originally put forward by Douady and Hubbard (see \cite{Do}) which produces quadratic rational maps by combining the dynamics of two quadratic polynomials.

Suppose $c \in \mathcal{L}$. Since $J_c$ is connected, $\mathbf{A}^\infty_c$ must be a simply connected domain. Let
\begin{displaymath}
\phi_c : \mathbf{A}^\infty_c \to \mathbb{D}
\end{displaymath}
be the unique conformal Riemann mapping such that $\phi_c(\infty) = 0$, and $\phi_c'(\infty) > 0$. It is not difficult to prove that the following diagram commutes:
\comdia{\mathbf{A}^\infty_c}{\mathbf{A}^\infty_c}{\mathbb{D}}{\mathbb{D}}{f_c}{z \mapsto z^2}{\phi_c}{\phi_c}
and hence, $\phi_c$ is the B\"ottcher uniformization of $f_c$ on $\mathbf{A}^\infty_c$. Moreover, since $J_c$ is locally connected, Carath\'eodory's theory implies that the inverse of $\phi_c$ extends continuously to the boundary of $\mathbb{D}$ (see \cite{M1}). If we let
\begin{displaymath}
\tau_c := \phi_c^{-1}|_{\partial \mathbb{D}},
\end{displaymath}
we obtain a continuous parametrization of $J_c$ by the unit circle $\partial \mathbb{D} = \mathbb{R} / \mathbb{Z}$ known as a \ebf{Carath\'eodory loop}. Observe that $f_c$, when restricted to $J_c$, acts via $\tau_c$ as the angle doubling map:
\begin{displaymath}
f_c(\tau_c(t)) = \tau_c(2t).
\end{displaymath}

Now, suppose $c_1, c_2 \in \mathcal{L}$. Using $\tau_{c_1}$ and $\tau_{c_2}$, we can glue the dynamics of $f_{c_1}$ and $f_{c_2}$ together to construct a new dynamical system as follows.

First we construct a new dynamical space $K_{c_1} \vee K_{c_2}$ by gluing the filled Julia sets $K_{c_1}$ and $K_{c_2}$:
\begin{equation} \label{gluing def}
K_{c_1} \vee K_{c_2} : = (K_{c_1} \sqcup K_{c_2}) / \{ \tau_{c_1}(t) \sim \tau_{c_2}(-t) \}.
\end{equation}
We refer to the resulting equivalence relation $\sim$ as \ebf{ray equivalence}, and denote it by $\sim_{\text{ray}}$. For $x \in K_{c_i}$, $i = 1, 2$, we denote the ray equivalency class of $x$ by $[x]_{\text{ray}}$.

We now define a new map
\begin{displaymath}
f_{c_1} \vee f_{c_2} : K_{c_1} \vee K_{c_2} \to K_{c_1} \vee K_{c_2},
\end{displaymath}
called the \ebf{formal mating of $f_{c_1}$ and $f_{c_2}$}, by letting $f_{c_1} \vee f_{c_2} \equiv f_{c_1}$ on $K_{c_1}$ and $f_{c_1} \vee f_{c_2} \equiv f_{c_2}$ on $K_{c_2}$.

\begin{figure}[h]
\centering
\includegraphics[scale=0.7]{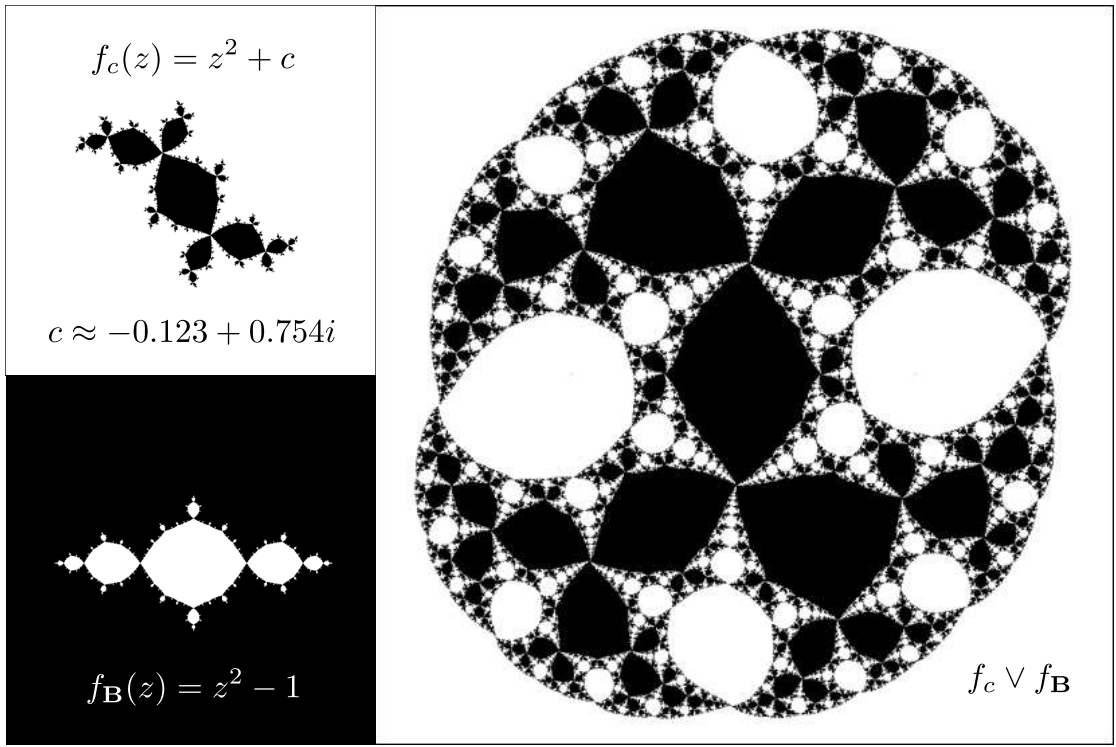}
\caption{The Douady rabbit $f_c$, $c \approx -0.123+0.754i$, mated with the basilica polynomial $f_\mathbf{B}$.}
\label{fig:douadyrabbit}
\end{figure}

If the space $K_{c_1} \vee K_{c_2}$ is homeomorphic to the 2-sphere, then $f_{c_1}$ and $f_{c_2}$ are said to be \ebf{topologically mateable}. If, in addition, there exists a quadratic rational map $R$ and a homeomorphism 
\begin{displaymath}
\Lambda : K_{c_1} \vee K_{c_2} \to \hat{\mathbb{C}}
\end{displaymath}
such that $\Lambda$ is conformal on $\mathring{K}_{c_1} \sqcup \mathring{K}_{c_2} \subset K_{c_1} \vee K_{c_2}$, and the following diagram commutes:
\begin{displaymath}
\begin{CD}
K_{c_1} \vee K_{c_2} @>f_{c_1} \vee f_{c_2}>> K_{c_1} \vee K_{c_2}\\
@VV\Lambda V @VV\Lambda V\\
\hat{\mathbb{C}} @>R>> \hat{\mathbb{C}}
\end{CD}
\end{displaymath}
then $f_{c_1}$ and $f_{c_2}$ are said to be \ebf{conformally mateable}. The quadratic rational map $R$ is called a \ebf{conformal mating of $f_{c_1}$ and $f_{c_2}$}. We also say that $R$ \ebf{realizes the conformal mating of $f_{c_1}$ and $f_{c_2}$}.

In applications, it is sometimes more useful to work with the following reformulation of the definition of conformal mateability:
 
\begin{prop} \label{mating def}
Suppose $c_1, c_2 \in \mathcal{L}$. Then $f_{c_1}$ and $f_{c_2}$ are conformally mateable if and only if there exists a pair of continuous maps
\begin{displaymath}
\Lambda_1 : K_{c_1} \to \hat{\mathbb{C}}, \hspace{5mm} \Lambda_2 : K_{c_2} \to \hat{\mathbb{C}}
\end{displaymath}
such that for all $i, j \in \{1, 2\}$ the following three conditions are satisfied:
\begin{enumerate}[label=(\roman{*})]
\item $\Lambda_i(z) = \Lambda_j(w)$ if and only if $z \sim_{\text{ray}} w$
\item $\Lambda_i$ is conformal on $\mathring{K}_{c_i}$
\item there exists a rational function $R$ of degree 2 such that the following diagrams commute:
\begin{displaymath}
\begin{CD}
K_{c_1} @>f_{c_1}>> K_{c_1}\\
@VV\Lambda_1 V @VV\Lambda_1 V\\
\hat{\mathbb{C}} @>R>> \hat{\mathbb{C}}
\end{CD}
\hspace{20 mm}
\begin{CD}
K_{c_2} @>f_{c_2}>> K_{c_2}\\
@VV\Lambda_2 V @VV\Lambda_2 V\\
\hat{\mathbb{C}} @>R>> \hat{\mathbb{C}}
\end{CD}
\end{displaymath}
\end{enumerate}
\end{prop}

\begin{cor}
Suppose $R$ is a conformal mating of $f_{c_1}$ and $f_{c_2}$ for some $c_1, c_2 \in \mathcal{L}$. Then $R$ has a locally connected Julia set $J(R)$.
\end{cor}

\begin{proof}
Let $\Lambda_1 : K_{c_1} \to \hat{\mathbb{C}}$ and $\Lambda_2 : K_{c_2} \to \hat{\mathbb{C}}$ be as given in proposition \ref{mating def}. Note that
\begin{displaymath}
J(R) = \Lambda_1(J_{c_1}) = \Lambda_2(J_{c_2}).
\end{displaymath}
Since the continuous image of a compact locally connected set is locally connected, the result follows.
\end{proof}

\begin{ex} \label{triv ex}
For every $c \in \mathcal{L}$, $f_c$ is trivially conformally mateable with the squaring map $f_0(z) = z^2$. This follows from choosing $\Lambda_1$ and $\Lambda_2$ in proposition \ref{mating def} to be the identity map on $K_c$ and the inverse of the B\"ottcher uniformization of $f_c$ on $\mathbf{A}^\infty_c$ respectively. Note that the conformal mating of $f_c$ and $f_0$ is given by $f_c$ itself.

The converse is given by the following easy result:

\begin{prop} \label{polytriv}
Suppose a quadratic polynomial $P : \mathbb{C} \to \mathbb{C}$ is a conformal mating of $f_{c_1}$ and $f_{c_2}$ for some $c_1, c_2 \in \mathcal{L}$. Then either $f_{c_1}$ or $f_{c_2}$ must be equal to the squaring map $f_0$.
\end{prop}

\begin{proof}
Let $J(P)$ and $\mathbf{A}^\infty_P$ denote the Julia set and the attracting basin of infinity for $P$ respectively. We have
\begin{displaymath}
J(P) = \Lambda_1(J_{c_1}) = \Lambda_2(J_{c_2}).
\end{displaymath}
Hence, $\mathbf{A}^\infty_P$ must be contained in either $\Lambda_1(\mathring{K}_{c_1})$ or $\Lambda_2(\mathring{K}_{c_2})$. Assume for concreteness that it is contained in the former. Since $\Lambda_1|_{\mathring{K}_{c_1}}$ is conformal, and
\begin{displaymath}
f_{c_1}(z) = \Lambda_1^{-1} \circ P \circ \Lambda_1(z)
\end{displaymath}
for all $z \in K_{c_1}$, we see that $\Lambda_1^{-1}(\infty)$ must be a superattracting fixed point for $f_{c_1}$. The only member in the quadratic family that has a bounded superattracting fixed point is the squaring map $f_0$.
\end{proof}

By proposition \ref{polytriv}, we see that except in the trivial case, the mating construction yields non-polynomial dynamical systems.
\end{ex}

\begin{ex} \label{no half limb}
Consider the formal mating of the basilica polynomial $f_\mathbf{B}(z):= f_{-1}(z) = z^2 -1$ with itself. The glued space $K_\mathbf{B} \vee K_\mathbf{B}$ consists of infinitely many spheres connected together at discrete nodal points (refer to section \ref{basilica bubble}). Hence, it is not homeomorphic to the 2-sphere. Therefore, $f_\mathbf{B}$ is not conformally mateable with itself (since it is not even topologically mateable with itself). This is actually a specific instance of a more general result, which we state below.

Let $H_0$ be the \ebf{principal hyperbolic component} defined as the set of $c\in\mathcal M$ for which $f_c$ has an attracting fixed point $z_c\in\mathbb C$. It is conformally parametrized by the multiplier of $z_c$:
\begin{displaymath}
\lambda : c \mapsto \mu_c := f_c'(z_c)
\end{displaymath}
(see e.g. \cite{M2}). Note that $\lambda$ extends to a homeomorphism between $\overline{H_0}$ and $\overline{\mathbb{D}}$.

A connected component of ${\mathcal M}\setminus \overline{H_0}$ is called a \ebf{limb}. It is known (see e.g. \cite{M2}) that the closure of every limb intersects $\partial H_0$ at a single point. Moreover, the image of this point under $\lambda$ is a root of unity. Henceforth, the limb growing from the point $\lambda^{-1}(e^{2\pi ip/q})$, $p/q\in\mathbb Q$, will be denoted by $L_{p/q}$. For example, the parameter value $-1$ for the basilica polynomial $f_\mathbf{B}(z)= z^2 -1$ is contained in the 1/2-limb $L_{1/2}$.

The following standard observation is due to Douady \cite{Do}:

\begin{prop}
Suppose $c_1$ and $c_2$ are contained in complex conjugate limbs $L_{p/q}$ and $L_{-p/q}$ of the Mandelbrot set $\mathcal{M}$. Then $f_{c_1}$ and $f_{c_2}$ are not topologically mateable.
\end{prop}

\begin{proof}
There exists a unique repelling fixed point $\alpha_1 \in K_1$ (resp. $\alpha_2 \in K_2$) such that $K_1\setminus \{\alpha_1\}$ (resp. $K_2 \setminus \{\alpha_2\}$) is disconnected. Since $c_1$ and $c_2$ are contained in complex conjugate limbs, $\alpha_1$ and $\alpha_2$ are in the same ray equivalency class. Hence they are glued together to a single point in $K_{c_1} \vee K_{c_2}$. Removing this single point from $K_{c_1} \vee K_{c_2}$ leaves it disconnected, which is impossible if $K_{c_1} \vee K_{c_2}$ is homeomorphic to the 2-sphere. For more details, see \cite{M2}.
\end{proof}
\end{ex}

\newpage

\section{Matings with the Basilica Polynomial} \label{motivate}

Matings can be particularly useful in describing the dynamics in certain one-parameter families of rational maps. The best studied example of such a family is
\begin{displaymath}
R_a(z) := \frac{a}{z^2 + 2z}, \hspace{5mm} a \in \mathbb{C} \setminus \{0\},
\end{displaymath}
which is referred to as the \ebf{basilica family}.

The critical points for $R_a$ are $\infty$ and $-1$. Observe that $\{\infty, 0\}$ is a superattracting 2-periodic orbit for $R_a$. Let $\mathcal{A}^\infty_a$ be the attracting basin of $\{\infty, 0\}$. The boundary of $\mathcal{A}^\infty_a$ is equal to the Julia set $J(R_a)$.

\begin{prop} \label{super 2orbit}
Suppose $f : \hat{\mathbb{C}} \to \hat{\mathbb{C}}$ is a quadratic rational map with a superattracting 2-periodic orbit. Then $f$ can be normalized as $R_a$ for some unique $a \in \mathbb{C} \setminus \{0\}$ by a linear change of coordinates.
\end{prop}

\begin{proof}
First we show that any quadratic rational map $f$ that has a superattracting 2-periodic orbit $\{\infty, 0\}$ with critical points at $\infty$ and $-1$ must be of the form $R_a$ for some unique $a \in \mathbb{C} \setminus \{0\}$.

To this end, let
\begin{displaymath}
f(z) = \frac{a_2 z^2 + a_1 z + a_0}{b_2 z^2 + b_1 z + b_0}.
\end{displaymath}
If $f(\infty) = 0$ then $a_2 = 0$ and $b_2 \neq 0$. If $f(0) = \infty$, then $a_0 \neq 0$ and $b_0 = 0$. Hence, $f$ can be uniquely expressed as
\begin{displaymath}
f(z) = \frac{a_1 z + a_0}{z^2 + b_1 z}.
\end{displaymath}

If $a_1 \neq 0$, then for $r$ sufficiently large,
\begin{displaymath}
f(r e^\theta) \sim \frac{a_1}{r} e^{-\theta}. 
\end{displaymath}
Which implies that $\infty$ cannot be a critical point for $f$ by the argument principle. Hence, $a_1 = 0$.

Finally, we compute
\begin{displaymath}
f'(-1) = \frac{2a_0 - a_0 b_1}{(1 - b_1)^2} = 0.
\end{displaymath}
Hence, $b_1 = 2$.

In the general case, suppose $f : \hat{\mathbb{C}} \to \hat{\mathbb{C}}$ has a superattracting 2-periodic orbit $\{z_\infty, z_0\}$ and critical points at $z_\infty$ and $z_{-1}$. Then there exists a unique linear change of coordinates which sends $z_\infty$ to $\infty$, $z_0$ to $0$, and $z_{-1}$ to $-1$. The result follows. 
\end{proof}

Analogously to $\mathcal{M}$, the non-escape locus in the parameter space for $R_a$ is defined as
\begin{displaymath}
\mathcal{M}_\mathbf{B} := \{a \in \mathbb{C} \setminus \{0\} \hspace{2mm} | \hspace{2mm} -1 \notin \mathcal{A}^\infty_a\}.
\end{displaymath}
We also define the following subset of $\mathcal{M}_\mathbf{B}$:
\begin{displaymath}
\mathcal{L}_\mathbf{B} :=  \{a \in \mathcal{M}_\mathbf{B} \hspace{2mm} | \hspace{2mm} J(R_a) \text{ is locally connected}\}.
\end{displaymath}

\begin{figure}[h]
\centering
\includegraphics[scale=0.5]{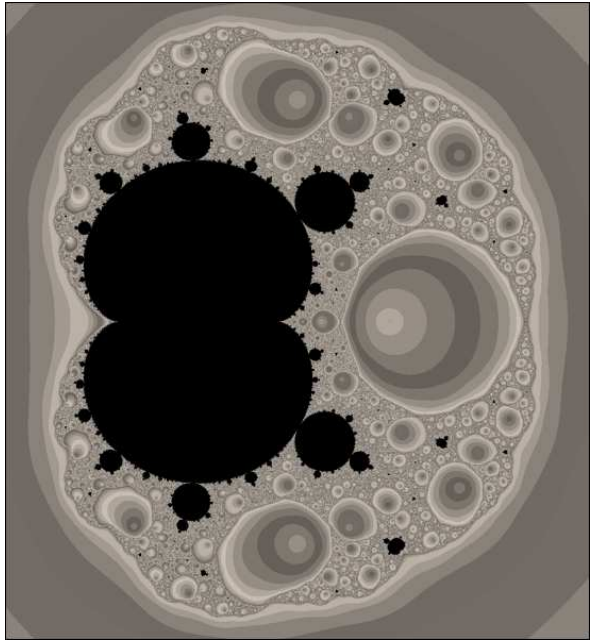}
\caption{The non-escape locus $\mathcal{M}_\mathbf{B}$ for $R_a$ (in black). Compare with figure \ref{fig:mandelbrot}. Note the absence of a copy of the 1/2-limb $L_{1/2}$ (see example \ref{no half limb}).}
\label{fig:matedmandelbrot}
\end{figure}

The basilica polynomial
\begin{displaymath}
f_\mathbf{B}(z) := z^2 -1
\end{displaymath}
is the only member of the quadratic family that has a superattracting 2-periodic orbit. Let $K_\mathbf{B}$ be the filled Julia set for $f_\mathbf{B}$. The following result is an analogue of the B\"ottcher uniformization theorem for the quadratic family. Refer to \cite{AY} for the proof.

\begin{prop} \label{basilica unif}
Suppose $a \in \mathcal{M}_\mathbf{B}$. Then there exists a unique conformal map $\psi_a : \mathcal{A}^\infty_a \to \mathring{K}_\mathbf{B}$ such that the following diagram commutes:
\comdia{\mathcal{A}^\infty_a}{\mathcal{A}^\infty_a}{\mathring{K}_\mathbf{B}}{\mathring{K}_\mathbf{B}}{R_a}{f_\mathbf{B}}{\psi_a}{\psi_a}
Moreover, if $B$ is a connected component of $\mathcal{A}^\infty_a$, then $\phi_a$ extends to a homeomorphism between $\overline{B}$ and $\phi_a(\overline{B})$.
\end{prop}

Suppose for some $c \in \mathcal{L} \cap (\mathbb{C} \setminus L_{1/2})$, $f_c$ and $f_\mathbf{B}$ are conformally mateable. If $F : \hat{\mathbb{C}} \to \hat{\mathbb{C}}$ is a conformal mating of $f_c$ and $f_\mathbf{B}$, then $F$ has a superattracting 2-periodic orbit. By proposition \ref{super 2orbit}, $F$ can be normalized as $R_a$ for some $a \in \mathcal{L}_\mathbf{B}$.

In view of proposition \ref{basilica unif}, it is natural to ask whether for every $a \in \mathcal{L}_\mathbf{B}$, $R_a$ is a conformal mating of $f_c$ and $f_\mathbf{B}$ for some $c \in \mathcal{L} \cap (\mathbb{C} \setminus L_{1/2})$. It turns out this cannot be true: for some $a \in \mathcal{L}_\mathbf{B}$, $R_a$ is the result of a more general form of mating called \emph{mating with laminations} between $f_c$ and $f_\mathbf{B}$ with $c \notin \mathcal{L}$ (see \cite{Du}). However, the following weaker statement does hold. The proof is completely analogous to the proof of proposition \ref{polytriv}, so we omit it here.

\begin{prop}
Suppose $R_a$ is a conformal mating. Then $R_a$ is a conformal mating of $f_c$ and $f_\mathbf{B}$ for some $c \in \mathcal{L} \cap (\mathbb{C} \setminus L_{1/2})$.
\end{prop}

The principal motivation for this paper is to answer the following question:

\begin{mainquestion}
Suppose $c \in \mathcal{L} \cap (\mathbb{C} \setminus L_{1/2})$. Are $f_c$ and $f_\mathbf{B}$ conformally mateable? If so, is there a unique member of the basilica family that realizes their conformal mating?
\end{mainquestion}

We now summarize the known results on this topic.

\begin{thm} [Rees, Tan, Shishikura \cite{Re, Tan, S}] \label{rees}
Suppose $c \in \mathcal{L} \cap (\mathbb{C} \setminus L_{1/2})$. If $f_c$ is hyperbolic, then $f_c$ and $f_\mathbf{B}$ are conformally mateable. Moreover, their conformal mating is unique up to conjugacy by a M\"obius map.
\end{thm}

Theorem \ref{rees} is actually a corollary of a much more general result which states that two post-critically finite quadratic polynomials $f_{c_1}$ and $f_{c_2}$ are (essentially) mateable if and only if $c_1$ and $c_2$ do not belong to conjugate limbs of the Mandelbrot set. See \cite{Tan} for more details.

\begin{thm} [Aspenberg, Yampolsky \cite{AY}] \label{asp}
Suppose $c \in \mathcal{L} \cap (\mathbb{C} \setminus L_{1/2})$. If $f_c$ is at most finitely renormalizable and has no non-repelling periodic orbits, then $f_c$ and $f_\mathbf{B}$ are conformally mateable. Moreover, their conformal mating is unique up to conjugacy by a M\"obius map.
\end{thm}

\begin{thm} [Dudko \cite{Du}] \label{dudko}
Suppose $c \in \mathcal{L} \cap (\mathbb{C} \setminus L_{1/2})$. If $f_c$ is at least 4 times renormalizable, then $f_c$ and $f_\mathbf{B}$ are conformally mateable. Moreover, their conformal mating is unique up to conjugacy by a M\"obius map.
\end{thm}

Together, theorem \ref{rees}, \ref{asp} and \ref{dudko} provide a positive answer to the main question in almost all cases. However, the parameters contained in the boundary of hyperbolic components that are not too ``deep'' inside the Mandelbrot set are still left unresolved. We discuss these parameters in greater detail in the next section.

\newpage

\section{Matings in the Boundary of Hyperbolic Components} \label{mainthms}

Let $H$ be a hyperbolic component of $\mathcal{M} \setminus L_{1/2}$. By theorem \ref{rees}, the quadratic polynomial $f_c$ and the basilica polynomial $f_\mathbf{B}$ are conformally mateable for all $c \in H$. Our goal is to determine if this is also true for $c \in \partial H \cap \mathcal{L}$.

Choose a parameter value $c_0 \in H$, and let $a_0 \in \mathcal{M}_\mathbf{B}$ be a parameter value such that $R_{a_0}$ is a conformal mating of $f_{c_0}$ and $f_\mathbf{B}$. Since $R_{a_0}$ must be hyperbolic, $a_0$ is contained in some hyperbolic component $H_\mathbf{B}$ of $\mathcal{M}_\mathbf{B}$.

For all $c \in \overline{H}$, $f_c$ has a non-repelling $n$-periodic orbit $\mathbf{O}_c := \{f_c^i(z_c)\}_{i=0}^{n-1}$ for some fixed $n \in \mathbb{N}$ (see e.g. \cite{M2}). Likewise, for all $a \in \overline{H_{\mathbf{B}}}$, $R_a$ has a non-repelling $n$-periodic orbit $\mathcal{O}_a := \{R_a^i(w_a)\}_{i=0}^{n-1}$. Define the multiplier maps $\lambda: \overline{H} \to \overline{\mathbb{D}}$ and $\mu : \overline{H_\mathbf{B}} \to \overline{\mathbb{D}}$ by:
\begin{displaymath}
\lambda(c) := (f_c^i)'(z_c) \hspace{2.5 mm} \text{and} \hspace{2.5 mm} \mu(a) := (R_a^i)'(w_a).
\end{displaymath}
It is known that $\lambda$ and $\mu$ are homeomorphisms which are conformal on the interior of their domains (see \cite{M2}).

The following result can be proved using a standard application of quasiconformal surgery (see chapter 4 in \cite{BF}).

\begin{prop} \label{hyper comp unif}
Define a homeomorphism $\phi_H : \overline{H} \to \overline{H_\mathbf{B}}$ by
\begin{displaymath}
\phi_H := \mu^{-1} \circ \lambda.
\end{displaymath}
Then for all $c \in H$, $R_{\phi_H(c)}$ is a conformal mating of $f_c$ and $f_\mathbf{B}$.
\end{prop}

Our goal is to extend this result to the boundary of $H$ where possible.

Consider $c \in \partial H$, and let $a = \phi_H(c) \in \partial H_\mathbf{B}$. The multiplier of $\mathbf{O}_c$ and $\mathcal{O}_a$ is equal to $e^{2\pi \theta i}$ for some $\theta \in \mathbb{R} / \mathbb{Z}$. The number $\theta$ is referred to as the \ebf{rotation number}. If $\theta$ is rational, then $\mathbf{O}_c$ and $\mathcal{O}_a$ are parabolic. In this case, an application of trans-quasiconformal surgery due to Ha\"issinsky implies the following result (see \cite{Ha}).

\begin{thm}
Suppose that the rotation number $\theta$ is rational, so that $\mathbf{O}_c$ and $\mathcal{O}_a$ are parabolic. Then $f_c$ and $f_\mathbf{B}$ are conformally mateable, and $R_a$ is the unique member of the basilica family that realizes their conformal mating.
\end{thm}

If $\theta$ is irrational, then $\mathbf{O}_c$ is either Siegel or Cremer. In the latter case, it is known that the Julia set $J_c$ for $f_c$ is non-locally connected (see e.g. \cite{M1}). This means that the formal mating of $f_c$ and $f_\mathbf{B}$ cannot be defined, and hence, they are not conformally mateable.

For our discussion of the Siegel case, we first recall a classical result of Siegel \cite{S}. An irrational number $x$ is said to be \ebf{Diophantine of order $\kappa$} if there exists a fixed constant $\epsilon > 0$ such that for all $\frac{p}{q} \subset \mathbb{Q}$, the following inequality holds:
\begin{displaymath}
|x - \frac{p}{q}| \geq \frac{\epsilon}{q^\kappa}.
\end{displaymath}
The set of all irrational numbers that are Diophantine of order $\kappa$ is denoted $\mathcal{D}(\kappa)$. The smallest possible value of $\kappa$ such that $\mathcal{D}(\kappa)$ is non-empty is $2$ (see \cite{M1}).

\begin{thm} [Siegel \cite{S}] \label{siegel thm}
Let $f : U \to V$ be an analytic function. Suppose $f$ has an indifferent periodic orbit  $\mathcal{O}$ with an irrational rotation number $\theta$. If $\theta \in \mathcal{D}(\kappa)$ for some $\kappa \geq 2$, then $\mathcal{O}$ is a Siegel orbit.
\end{thm}

There is a classical connection between Diophantine classes and continued fraction approximations (see e.g. \cite{M1}). In particular, if
\begin{displaymath}
x = \cfrac{1}{a_1+\cfrac{1}{a_2+ \ldots{}}}
\end{displaymath}
is the continued fraction representation of $x$, then $x \in \mathcal{D}(2)$ if and only if all the $a_i$'s are uniformly bounded. In view of this, we say that the numbers $x \in \mathcal{D}(2)$ are of \ebf{bounded type}. Siegel quadratic polynomials of bounded type are prominently featured in the study of renormalization (see e.g. \cite{P, Mc, Y1, Y2}).

\begin{thm} [Peterson \cite{P}] \label{bounded type lc}
Suppose a quadratic polynomial $f_c$ has an indifferent periodic orbit with an irrational rotation number of bounded type. Then $f_c$ has a locally connected Julia set $J_c$.
\end{thm}

In this paper, we present a positive answer to the main question (stated in section \ref{motivate}) for quadratic polynomials $f_\mathbf{S}$ that have an indifferent fixed point with an irrational rotation number of bounded type. Note that by theorem \ref{siegel thm}, the indifferent fixed point is Siegel, and by theorem \ref{bounded type lc}, the formal mating of $f_\mathbf{S}$ and $f_\mathbf{B}$ is well defined.

The solution to the uniqueness part of the main question is elementary.

\begin{prop} \label{uniqueness mat}
Suppose $\lambda \in \overline{\mathbb{D}}$. Then there exists a unique $c \in \mathcal{M}$ (resp. $a \in \mathcal{M}_\mathbf{B}$) such that $f_c$ (resp. $R_a$) has a bounded non-repelling fixed point $z_0 \neq \infty$ with multiplier $\lambda$.
\end{prop}

\begin{proof}
Suppose $f_c$ has a fixed point $z_0 \neq \infty$ with multiplier $\lambda \in \mathbb{C}$. It is easy to check that the value of $c$ is given by
\begin{displaymath}
c = \frac{\lambda}{2} - \frac{\lambda^2}{4}.
\end{displaymath}
Hence, $c$ is uniquely determined.

Likewise, suppose $R_a$ has a fixed point with multiplier $\lambda \in \mathbb{C}$. Then the value of $a$ is given by
\begin{displaymath}
a = -\frac{8\lambda}{(\lambda-1)^3}.
\end{displaymath}
Hence, $a$ is uniquely determined.
\end{proof}

Our main results are stated below.

\begin{thma}
Suppose $\nu \in \mathbb{R} \setminus \mathbb{Q}$ is of bounded type. Let $R_\nu$ be the unique member of the basilica family that has a Siegel fixed point $z_0$ with rotation number $\nu$. Let $\mathcal{S}_0$ be the fixed Siegel disc containing $z_0$. Then $\mathcal{S}_0$ is a quasidisk, and contains a unique critical point in its boundary.
\end{thma}

\begin{thmb}
Suppose $\nu \in \mathbb{R} \setminus \mathbb{Q}$ is of bounded type. Let $f_\mathbf{S}$ be the unique member of the quadratic family that has a Siegel fixed point with rotation number $\nu$. Then $f_\mathbf{S}$ and $f_\mathbf{B}$ are conformally mateable, and $R_\nu$ is the unique member of the basilica family that realizes their conformal mating.
\end{thmb}

\begin{figure}[h]
\centering
\includegraphics[scale=0.7]{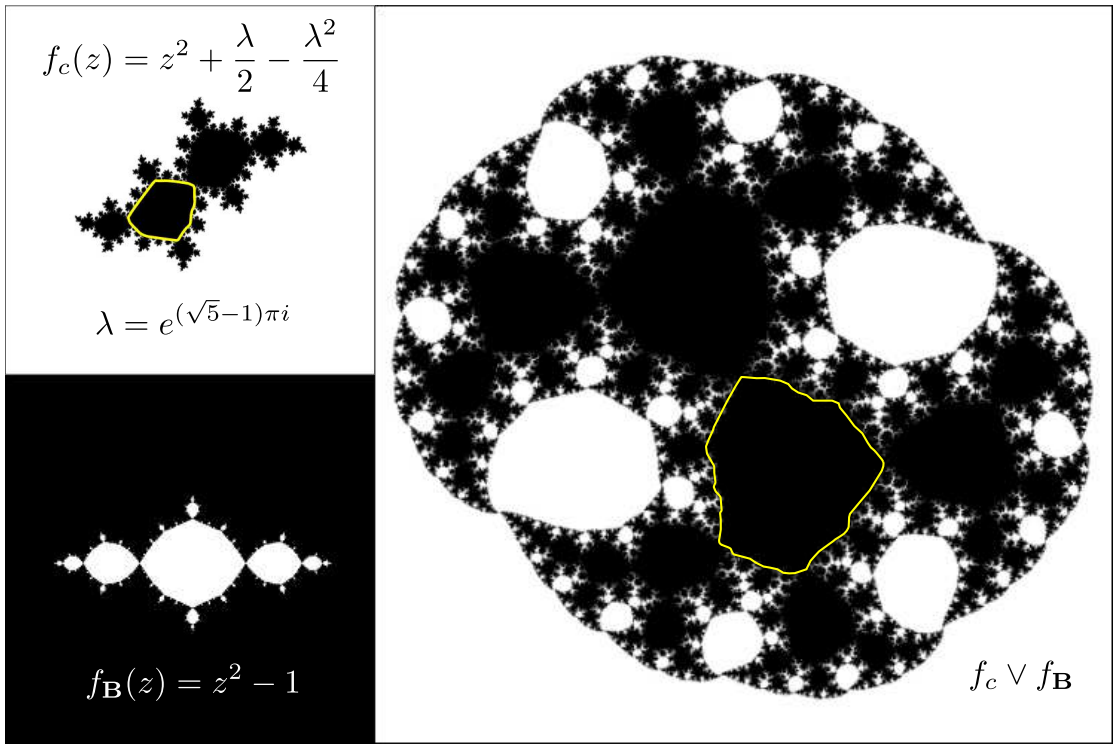}
\caption{Mating of a Siegel polynomial $f_c$, $c = \frac{\lambda}{2} - \frac{\lambda^2}{4}$, $\lambda = e^{(\sqrt{5}-1) \pi i}$, and the basilica polynomial $f_\mathbf{B}$. The Siegel disc is highlighted.}
\label{fig:matingsiegelbasilica}
\end{figure}

\newpage

\section{The Construction of a Blaschke Product Model and the Proof of Theorem A}

Suppose $\nu \in \mathbb{R} \setminus \mathbb{Q}$ is of bounded type. It follows from proposition \ref{uniqueness mat} that there exists a unique member of the basilica family $R_\nu$ that has a fixed Siegel disc $\mathcal{S}_0$ with rotation number $\nu$. In this section, we use quasiconformal surgery to show that $\mathcal{S}_0$ is a quasidisc. Refer to \cite{BF} for generalities about quasiconformal maps and quasiconformal surgeries.

Consider the Blaschke product
\begin{displaymath}
F_{a,b}(z) := -\frac{1}{e^{i \theta}} \frac{z(z-a)(z-b)}{(1-\bar{a}z)(1-\bar{b}z)},
\end{displaymath}
where
\begin{displaymath}
ab = r e^{i \theta}, \hspace{5 mm} r \in \mathbb{R}^+, \hspace{2 mm} \theta \in [0, 2 \pi).
\end{displaymath}
Note that $0$ is a fixed point with multiplier $-r$.

\begin{lem}\label{crit}
For any value of $r$ and $\theta$, $a = a(r, \theta)$ and $b = b(r, \theta)$ can be chosen such that $F_{a,b}$ has a double critical point at $1$.
\end{lem}

\begin{proof}
Let
\begin{displaymath}
F_{a,b}'(z) = \frac{P(z)}{Q(z)}.
\end{displaymath}
Then
\begin{displaymath}
F_{a,b}''(z) = \frac{P'(z)Q(z) - P(z)Q'(z)}{Q(z)^2}.
\end{displaymath}
Thus, the condition
\begin{displaymath}
F_{a,b}'(1) = F_{a,b}''(1) = 0
\end{displaymath}
is equivalent to
\begin{displaymath}
P(1) = P'(1) = 0.
\end{displaymath}

A straightforward computation shows that
\begin{displaymath}
P(z) = \overline{\kappa}z^4 - 2 \overline{\zeta}z^3 + (3 - |\kappa|^2 + |\zeta|^2)z^2 - 2 \zeta z + \kappa,
\end{displaymath}
where
\begin{displaymath}
\kappa := ab, \hspace{5 mm} \zeta := a+b.
\end{displaymath}
Thus, $F_{a,b}$ has a double critical point at $1$ if the following two equations are satisfied:
\begin{equation} \label{1}
2\kappa - 3 \zeta + (3 - |\kappa|^2 + |\zeta|^2) = \overline{\zeta}
\end{equation}
\begin{equation} \label{2}
3\kappa - 2 \zeta + (3 - |\kappa|^2 + |\zeta|^2) = \overline{\kappa}.
\end{equation}

Subtracting \eqref{1} from \eqref{2}, we see that
\begin{displaymath}
\kappa - \zeta = \overline{\kappa} - \overline{\zeta}.
\end{displaymath}
Substituting $\kappa = x + i y$ and $\zeta = u + i y$ into \eqref{1}, we obtain
\begin{equation} \label{3}
u^2 -4u+(2x-x^2+3)=0.
\end{equation}
\eqref{3} has two solutions: $u = -x+3$ and $u = x+1$. The first solution corresponds to the relation
\begin{displaymath}
\zeta = -\overline{\kappa} + 3.
\end{displaymath}
Therefore, by choosing $a$ and $b$ to be the solutions of
\begin{displaymath}
z^2 + (re^{-i\theta}-3)z + re^{i\theta} = 0,
\end{displaymath}
we ensure that the map $F_{a,b}$ has a double critical point at $1$.
\end{proof}

\begin{lem} \label{conv coord}
Let $a = a(r, \theta)$ and $b = b(r, \theta)$ satisfy the condition in lemma \ref{crit}. Then for all $r > 1$ sufficiently close to $1$, there exists a local holomorphic change of coordinates $\phi$ at $0$ so that the map $G := \phi^{-1} \circ F_{a,b}^2 \circ \phi$ takes the form
\begin{displaymath}
G(z) = r^2 z (1 + z^2 + \mathcal{O}(z)).
\end{displaymath}
\end{lem}

\begin{proof}
Expanding $F_{a,b}(z)$ as a power series around $0$, we have
\begin{displaymath}
F_{a,b}(z) = -rz + \lambda z^2 + \mathcal{O}(z^3)
\end{displaymath}
for some $\lambda = \lambda(r, \theta)$ depending continuously on $r$ and $\theta$. Define
\begin{displaymath}
\psi_\mu(z) := z + \mu z^2, \hspace{5mm} \mu \in \mathbb{C}.
\end{displaymath}
A straightforward computation shows that
\begin{displaymath}
H(z) := \psi_\mu^{-1} \circ F_{a,b} \circ \psi_\mu (z) = -rz + (\lambda + (1+r) \mu) z^2 + \mathcal{O}(z^3).
\end{displaymath}
Thus, by choosing $\mu = \frac{-\lambda}{1+r}$, we have
\begin{displaymath}
H(z) = -rz(1 + \nu z^2 + \mathcal{O}(z^3))
\end{displaymath}
for some $\nu = \nu(r, \theta)$ depending continuously on $r$ and $\theta$.

Observe that the second iterate of $H$ is equal to
\begin{displaymath}
H^2(z) = r^2 z (1 + (1+r^2) \nu z^2 + \mathcal{O}(z^3)).
\end{displaymath}
When $r = 1$, $0$ is a parabolic fixed point of multiplicity $2$. This means that $\nu(1, \theta)$ cannot be equal to zero for all $\theta \in [0, 2\pi)$. Hence, for some $\epsilon>0$ sufficiently small, $\nu(r, \theta)$ is not equal to zero for all $r \in (1 ,  1+\epsilon)$ and $\theta \in [0, 2\pi)$. After one more change of coordinates, we arrive at
\begin{displaymath}
G(z) := \sqrt{(1+r^2)\nu} \cdot H^2(\frac{z}{\sqrt{(1+r^2)\nu}}) = r^2 z (1 + z^2 + \mathcal{O}(z^3)).
\end{displaymath}
\end{proof}

\begin{lem} \label{attract}
Let $a = a(r, \theta)$ and $b = b(r, \theta)$ satisfy the condition in lemma \ref{crit}. Then for all $r > 1$ sufficiently close to $1$, $F_{a,b}$ has an attracting 2-periodic orbit near $0$.
\end{lem}

\begin{proof}
Consider the map $G := \phi^{-1} \circ F_{a,b}^2 \circ \phi$ defined in lemma \ref{conv coord}. We prove that $G$ has two attracting fixed points near $0$.

Observe that $G$ satisfies
\begin{displaymath}
|G(z)| = r^2 |z| (1+ \text{Re}(z^2) + (\text{higher terms}))
\end{displaymath}
and
\begin{displaymath}
\text{arg}(G(z)) = \text{arg}(z) + \text{Im}(z^2) + (\text{higher terms}).
\end{displaymath}
Consider the wedge shaped regions
\begin{displaymath}
V^+_\epsilon := \{\rho e^{2\pi i t} \in \mathbb{C} \hspace{2mm} | \hspace{2mm} 0 \leq \rho \leq \epsilon, \frac{3}{16} \leq t \leq \frac{5}{16}\}
\end{displaymath}
and
\begin{displaymath}
V^-_\epsilon := - V^+_\epsilon.
\end{displaymath}
It is easily checked that $G(V^+_\epsilon) \subset V^+_\epsilon$ and $G(V^-_\epsilon) \subset V^-_\epsilon$. Since $0$ is the only fixed point on the boundary of these regions, and it is repelling, $V^+_\epsilon$ and $V^-_\epsilon$ must each contain an attracting fixed point for $G$.
\end{proof}

\begin{thm} \label{fnu}
Given any angle $\nu \in [0, 2 \pi)$, there exists a Blaschke product $F_\nu$ that satisfy the following three properties:
\begin{enumerate}[label=(\roman{*})]
\item $F_\nu$ has a superattracting 2-periodic orbit $\mathcal{O} = \{\infty, F_\nu(\infty)\}$ with $F_\nu'(\infty) = 0$.
\item $F_\nu$ has a double critical point at $1$.
\item The rotation number of the map $F_\nu |_{\partial \mathbb{D}}$ is equal to $\nu$.
\end{enumerate}
\end{thm}

\begin{proof}
The family of Blaschke products $\{F_{a,b}\}$ that satisfy lemma \ref{crit} and \ref{attract} are continuously parameterized by $r$ and $\theta$. Let $\rho(r, \theta)$ denote the rotation number of the map $F_{a,b}|_{\partial \mathbb{D}}$. In \cite{YZ}, it is proved that $\rho(1, \cdot)$ is not nullhomotopic. By continuity, $\rho(r, \cdot)$ is also not nullhomotopic. Thus, for any angle $\nu \in [0, 2 \pi)$, there exists $\theta$ such that $\rho(r, \theta) = \nu$.

So far, we have proved the existence of a Blaschke product $F_{a, b}$ that has an attracting 2-periodic orbit near zero, has a double critical point at $1$, and whose restriction to $\partial \mathbb{D}$ has rotation number equal to $\nu$. A standard application of quasiconformal surgery turns the attracting 2-periodic orbits of $F_{a,b}$ into superattracting orbits (the surgery must be symmetric with respect to the unit circle to ensure that the resulting map is also a Blaschke product). Then after conjugating by the appropriate Blaschke factor, we obtain the desired map $F_\nu$.
\end{proof}

\begin{thm} \label{mod blaschke}
Suppose $\nu$ is irrational and of bounded type. Let $F_\nu$ be the Blaschke product constructed in theorem \ref{fnu}. Then there exists a quadratic rational function $R_\nu$ and quasiconformal maps $\psi : \mathbb{D} \to \mathbb{D}$, and $\phi : \hat{\mathbb{C}} \to \hat{\mathbb{C}}$ such that $\psi$ fixes $0$ and $1$; $\phi$ fixes $0$, $1$ and $\infty$; and
\begin{displaymath}
   R_\nu(z) = \left\{
     \begin{array}{ll}
      \phi \circ \psi^{-1} \circ \emph{Rot}_\nu \circ \psi \circ \phi^{-1}(z) & : \text{if } z \in \phi(\mathbb{D})\\
       \phi \circ F_\nu \circ \phi^{-1}(z) & : \text{if } z \in \hat{\mathbb{C}} \setminus \phi(\mathbb{D}).
     \end{array}
   \right.
\end{displaymath}
\end{thm}

\begin{proof}
Since $\nu$ is of bounded type, there exists a unique homeomorphism $\psi : \partial \mathbb{D} \to \partial \mathbb{D}$ such that $\psi(1) = 1$, and
\begin{displaymath}
\psi^{-1} \circ \text{Rot}_\nu \circ \psi = F_\nu|_{\partial \mathbb{D}}.
\end{displaymath}
Moreover, $\psi$ extends to a quasiconformal map on $\mathbb{D}$.

Define
\begin{displaymath}
   g(z) = \left\{
     \begin{array}{ll}
      \psi^{-1} \circ \text{Rot}_\nu \circ \psi (z) & : \text{if } z \in \mathbb{D}\\
       F_\nu(z) & : \text{if } z \in \hat{\mathbb{C}} \setminus \mathbb{D}.
     \end{array}
   \right.
\end{displaymath}
By construction, $g$ is continuous.

To obtain a holomorphic map with the same dynamics as $g$, we define and integrate a new complex structure $\mu$ on $\hat{\mathbb{C}}$. Start by defining $\mu$ on $\mathbb{D}$ as the pull back of the standard complex structure $\sigma_0$ by $\psi$. Next, pull back $\mu$ on $\mathbb{D}$ by the iterates of $g$ to define $\mu$ on the iterated preimages of $\mathbb{D}$. Finally, extend $\mu$ to the rest of $\hat{\mathbb{C}}$ as the standard complex structure $\sigma_0$.

Let $\phi : \hat{\mathbb{C}} \to \hat{\mathbb{C}}$ be the unique solution of the Beltrami equation
\begin{displaymath}
\partial_{\overline{z}} \phi(z) = \mu(z) \partial_z \phi(z)
\end{displaymath}
for which $\phi$ is a quasiconformal map fixing the points $0$, $1$ and $\infty$. The map
\begin{displaymath}
R_\nu := \phi \circ g \circ \phi^{-1}
\end{displaymath}
gives us the desired quadratic rational function.
\end{proof}

\begin{figure}[h]
\centering
\includegraphics[scale=0.7]{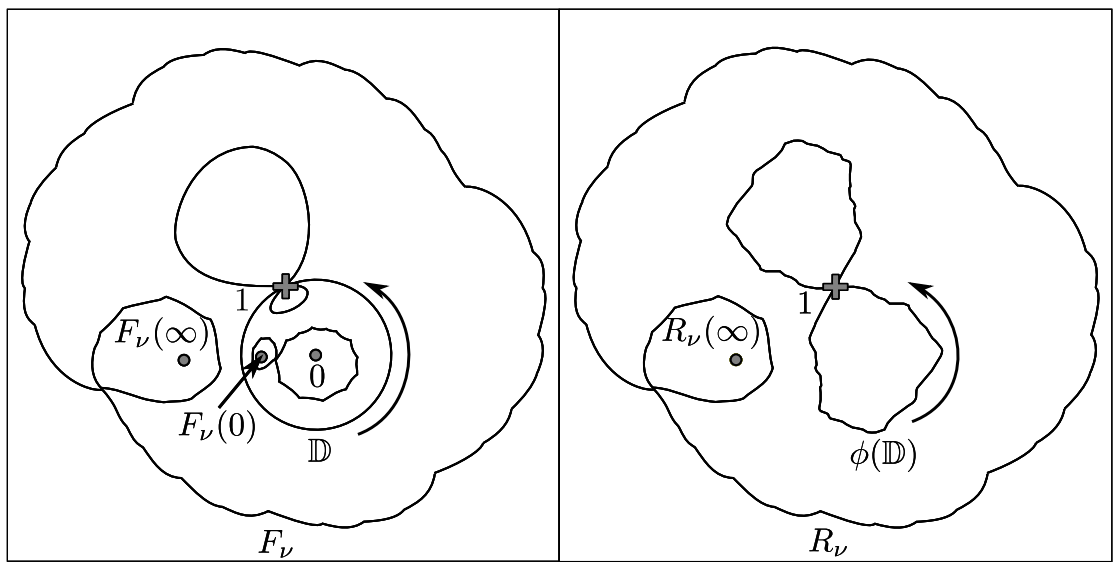}
\caption{Illustration of the quasiconformal surgery in theorem \ref{mod blaschke}.}
\label{fig:blaschkeconstruction}
\end{figure}

Theorem A stated in \ref{mainthms} now follows as a corollary of theorem \ref{mod blaschke}.

\newpage

\section{The Construction of Bubble Rays}

\subsection{For the basilica polynomial} \label{basilica bubble} \hspace{1mm}

\vspace{2.5mm}

Consider the basilica polynomial
\begin{displaymath}
f_\mathbf{B} := z^2 -1.
\end{displaymath}
$f_\mathbf{B}$ has a superattracting 2-periodic orbit $\{0, -1\}$, and hence, is hyperbolic. Denote the Julia set and the filled Julia set for $f_\mathbf{B}$ by $J_\mathbf{B}$ and $K_\mathbf{B}$ respectively. The following is a consequence of the hyperbolicity of $f_\mathbf{B}$ (see e.g. \cite{M1}).

\begin{prop}
The Julia set $J_\mathbf{B}$ for $f_\mathbf{B}$ is locally connected.
\end{prop}

A connected component of $\mathbf{B} := \mathring{K_\mathbf{B}}$ is called a \ebf{bubble}. Let $\mathbf{B}_0$ be the bubble containing the critical point $0$. We have
\begin{displaymath}
\mathbf{B} = \bigcup_{n=0}^\infty f_\mathbf{B}^{-n}(\mathbf{B}_0).
\end{displaymath}
Let $B \subset \mathbf{B}$ be a bubble. The \ebf{generation of} $B$, denoted by gen$(B)$, is defined to be the smallest number $n \in \mathbb{N}$ such that $f_\mathbf{B}^n(B) = \mathbf{B}_0$.

\begin{prop} \label{b fixed}
There exists a unique repelling fixed point $\mathbf{b}$ contained in $\partial \mathbf{B}_0$.
\end{prop}

Note that the repelling fixed point $\mathbf{b}$ in proposition \ref{b fixed} is the $\alpha$-fixed point of $f_\mathbf{B}$ (see \cite{M2}).

Let $b \in J_\mathbf{B}$ be an iterated preimage of $\mathbf{b}$. The \ebf{generation of} $b$, denoted by gen$(b)$, is defined to be the smallest number $n \in \mathbb{N}$ such that $f_\mathbf{B}^n(b) = \mathbf{b}$. Suppose $b$ is contained in the boundary of some bubble $B$. If the generation of $b$ is the smallest among all iterated preimages of $\mathbf{b}$ that are contained in $\partial B$, then $b$ is called the \ebf{root of $B$}.

\begin{prop} \label{parentchild}
Let $b \in J_\mathbf{B}$ be an iterated preimage of $\mathbf{b}$. Then there are exactly two bubbles $B_1$ and $B_2$ in $\mathbf{B}$ which contain $b$ in their closures.
\end{prop}

\begin{proof}
There are exactly two bubbles, $\mathbf{B}_0$ and $f_\mathbf{B}(\mathbf{B}_0)$, that contain $\mathbf{b}$ in their closure. There exists a neighbourhood $N$ containing $b$ such that $N$ is mapped conformally onto a neighbourhood of $\mathbf{b}$ by $f_\mathbf{B}^{\text{gen}(b)}$. The result follows.
\end{proof}

Let $b \in J_\mathbf{B}$ be an iterated preimage of $\mathbf{b}$, and let $B_1$ and $B_2$ be the two bubbles that contain $b$ in their closures. Suppose gen$(B_1) > $ gen$(B_2)$. Then $B_1$ and $B_2$ are referred to as the \ebf{parent} and the \ebf{child at $b$} respectively. It is easy to see that $b$ must be the root of $B_2$.

Consider a set of bubbles $\{B_i\}_{i=0}^n$ in $\mathbf{B}$, and a set of iterated preimages $\{b_i\}_{i=1}^n$ of $\mathbf{b}$ such that the following properties are satisfied:
\begin{enumerate}[label=(\roman{*})]
\item $B_0 = \mathbf{B}_0$.
\item $B_i$ and $B_{i+1}$ are the parent and the child at $b_{i+1}$ respectively.
\end{enumerate}
The set
\begin{displaymath}
\mathcal{R}^\mathbf{B} := \overline{f_\mathbf{B}(\mathbf{B}_0)} \cup (\bigcup \limits_{i=0}^n \overline{B_i})
\end{displaymath}
is called a \ebf{bubble ray for $f_\mathbf{B}$} (the inclusion of $\overline{f_\mathbf{B}(\mathbf{B}_0)}$ is to ensure that a bubble ray is mapped to a bubble ray). For conciseness, we use the notation $\mathcal{R}^\mathbf{B} \sim \{B_i\}_{i=0}^n$. $\mathcal{R}^\mathbf{B}$ is said to be \ebf{finite} or \ebf{infinite} according to whether $n < \infty$ or $n = \infty$. Lastly, $\{b_i\}_{i=1}^n$ is called the \ebf{set of attachment points for $\mathcal{R}^\mathbf{B}$}.

Let $\mathcal{R}^{\mathbf{B}} \sim \{B_i\}_{i=0}^\infty$ be an infinite bubble ray. We say that \ebf{$\mathcal{R}^\mathbf{B}$ lands at $z \in J_\mathbf{B}$} if the sequence of bubbles $\{B_i\}_{i=0}^\infty$ converges to $z$ in the Hausdorff topology. The following result is a consequence of the hyperbolicity of $f_\mathbf{B}$ (see \cite{DH}).

\begin{prop} \label{basilica landing}
There exists $0 < s < 1$, and $C >0$ such that for every bubble $B \subset \mathbf{B}$, we have
\begin{displaymath}
\emph{diam}(B) < C s^{\emph{gen}(B)}.
\end{displaymath}
Consequently, every infinite bubble ray for $f_\mathbf{B}$ lands.
\end{prop}

Denote the attracting basin of infinity for $f_\mathbf{B}$ by $\mathbf{A}_\mathbf{B}^\infty$. Let
\begin{displaymath}
\phi_{\mathbf{A}_\mathbf{B}^\infty} : \mathbf{A}_\mathbf{B}^\infty \to \hat{\mathbb{C}} \setminus \overline{\mathbb{D}}
\end{displaymath}
and
\begin{displaymath}
\phi_{\mathbf{B}_0} : \mathbf{B}_0 \to \mathbb{D}
\end{displaymath}
be the B\"ottcher uniformization of $f_\mathbf{B}$ on $\mathbf{A}_\mathbf{B}^\infty$ and $\mathbf{B}_0$ respectively. Using $\phi_{\mathbf{A}_\mathbf{B}^\infty}$ and $\phi_{\mathbf{B}_0}$, we can encode the dynamics of bubble rays for $f_\mathbf{B}$ in two different ways: via external angles, and via bubble addresses.

Suppose that $\mathcal{R}^{\mathbf{B}}$ is an infinite bubble ray, and let $z \in J_\mathbf{B}$ be its landing point. Then there exists a unique external ray
\begin{displaymath}
\mathcal{R}^\infty_{-t} := \{\text{arg}(\phi_{\mathbf{A}_\mathbf{B}^\infty}) = -t\}
\end{displaymath}
which lands on $z$. The \ebf{external angle of $\mathcal{R}^\mathbf{B}$} is defined to be $t$. Henceforth, the infinite bubble ray with external angle $t$ will be denoted $\mathcal{R}^\mathbf{B}_t$.

Let $b \in \partial\mathbf{B}_0$ be an iterated preimage of $\mathbf{b}$. Define
\begin{displaymath}
\text{adr}(b) := \text{ arg}(\phi_{\mathbf{B}_0}(b)).
\end{displaymath}
If $b'$ is an interated preimage of $\mathbf{b}$, and $b' \notin \partial \mathbf{B}_0$, then there exists a unique bubble $B \subset \mathbf{B}$ such that $B$ is the parent at $b$. In this case, define
\begin{displaymath}
\text{adr}(b') := \text{ adr}(f_\mathbf{B}^{\text{gen}(B)}(b')).
\end{displaymath}

Let $\mathcal{R}^\mathbf{B}$ be a bubble ray and let $\{b_i\}_{i=0}^n$ be the set of attachment points for $\mathcal{R}^{\mathbf{B}}$. The \ebf{bubble address of} $\mathcal{R}^\mathbf{B}$ is defined to be
\begin{displaymath}
\text{adr}(\mathcal{R}^\mathbf{B}) := (\text{adr}(b_1), \text{ adr}(b_2), \text{ }\ldots{}, \text{ adr}(b_n)),
\end{displaymath}
where the tuple is interpreted to be infinite if $\mathcal{R}^\mathbf{B}$ is an infinite bubble ray.

If $B \subset \mathbf{B}$ is a bubble, then there exists a unique finite bubble ray $\mathcal{R}^\mathbf{B} \sim \{B_i\}_{i=0}^n$ such that $B = B_n$. The \ebf{bubble address of $B$} is defined to be
\begin{displaymath}
\text{adr}(B) := \text{ adr}(\mathcal{R}^\mathbf{B}).
\end{displaymath}

\begin{figure}[h]
\centering
\includegraphics[scale=0.6]{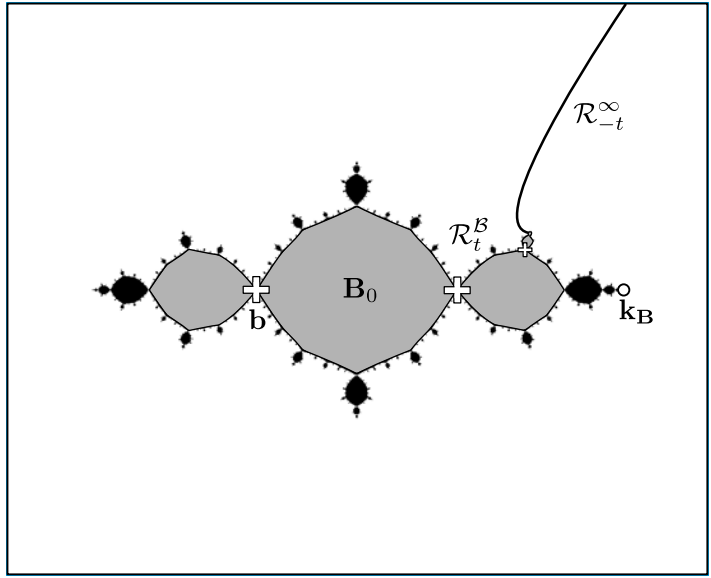}
\caption{The infinite bubble ray $\mathcal{R}^\mathbf{B}_t$, $t \approx -0.143$, for the basilica polynomial $f_\mathbf{B}$.}
\label{fig:basilicaextray}
\end{figure}

\subsection{For the Siegel polynomial} \label{siegel bubble} \hspace{1mm}

\vspace{2.5mm}

Suppose $\nu \in \mathbb{R} \setminus \mathbb{Q}$ is of bounded type, and let $f_\mathbf{S}$ be the unique member of the quadratic family that has a fixed Siegel disc $\mathbf{S}_0$ with rotation number $\nu$. Denote the Julia set and the filled Julia set for $f_\mathbf{S}$ by $J_\mathbf{S}$ and $K_\mathbf{S}$ respectively. By proposition \ref{bounded type lc}, $J_\mathbf{S}$ is locally connected. A quasiconformal surgery procedure due to Douady, Ghys, Herman, and Shishikura (see e.g. \cite{P}) implies the following:

\begin{thm}
The Siegel disc $\mathbf{S}_0$ is a quasidisc whose boundary contains the critical point $0$.
\end{thm}

A connected component of $\mathbf{S} := \mathring{K_\mathbf{S}}$ is called a \ebf{bubble}. Note that
\begin{displaymath}
\mathbf{S} = \bigcup_{n=0}^\infty f_\mathcal{S}^{-n}(\mathbf{S}_0).
\end{displaymath}
Let $S \subset \mathbf{S}$ be a bubble. The \ebf{generation of} $S$, denoted by gen$(S)$, is defined to be the smallest number $n \in \mathbb{N}$ such that $f_\mathbf{S}^n(S) = \mathbf{S}_0$. Similarly, let $s \in J_\mathbf{S}$ be an iterated preimage of $0$. The \ebf{generation of} $s$, denoted by gen$(s)$, is defined to be the smallest number $n \in \mathbb{N}$ such that $f_\mathbf{S}^n(s) = 0$. 

\begin{prop} \label{siegel preimage 0}
Let $s \in J_\mathbf{S}$ be an iterated preimage of the critical point $0$. Then there are exactly two bubbles $S_1$ and $S_2$ in $\mathbf{S}$ which contain $s$ in their closure.
\end{prop}

The construction of a \ebf{bubble ray $\mathcal{R}^\mathbf{S}$} for $f_\mathbf{S}$ is completely analogous to the construction of a bubble ray $\mathcal{R}^\mathbf{B}$ for $f_\mathbf{B}$. The following proposition is a consequence of complex a priori bounds due to Yampolsky (see \cite{Y1}). It is proved in the same way as proposition \ref{everylanding}.

\begin{prop} \label{siegel landing}
Every infinite bubble ray for $f_\mathbf{B}$ lands.
\end{prop}

Denote the attracting basin of infinity for $f_\mathbf{S}$ by $\mathbf{A}_\mathbf{S}^\infty$. Let
\begin{displaymath}
\phi_{\mathbf{A}_\mathbf{S}^\infty} : \mathbf{A}_\mathbf{S}^\infty \to \hat{\mathbb{C}} \setminus \overline{\mathbb{D}}
\end{displaymath}
be the B\"ottcher uniformization of $f_\mathbf{S}$ on $\mathbf{A}_\mathbf{S}^\infty$.

Suppose $\mathcal{R}^\mathbf{S}$ is an infinite bubble ray, and let $z \in J_\mathbf{S}$ be its landing point. Then there exists a unique external ray
\begin{displaymath}
\mathcal{R}^\infty_t := \{\text{arg}(\phi_{\mathbf{A}_\mathbf{S}^\infty}) = t\}
\end{displaymath}
which lands on $z$. The \ebf{external angle of $\mathcal{R}^\mathbf{S}$} is defined to be $t$. Henceforth, the infinite bubble ray with external angle $t$ will be denoted $\mathcal{R}^\mathbf{S}_t$.

Let $s \in \partial\mathbf{S}_0$ be an iterated preimage of $0$. Define
\begin{displaymath}
\text{adr}(s) := \text{ gen}(s).
\end{displaymath}
The \ebf{bubble address of a bubble $S \subset \mathbf{S}$} for $f_\mathbf{S}$ can now be defined in the same way as its counterpart for $f_\mathbf{B}$.

\begin{figure}[h]
\centering
\includegraphics[scale=0.6]{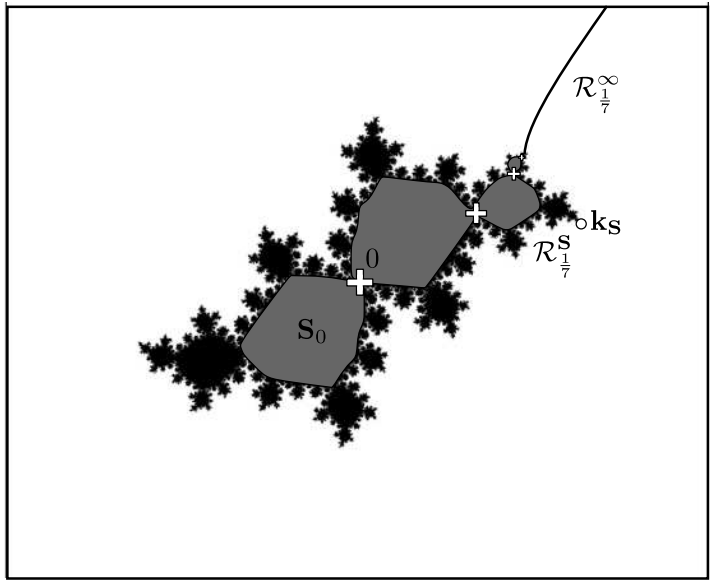}
\caption{The infinite bubble ray $\mathcal{R}^\mathbf{S}_{\frac{1}{7}}$ for the Siegel polynomial $f_\mathbf{S}$.}
\label{fig:siegelextray}
\end{figure}

\subsection{For the candidate mating} \label{candidate bubble} \hspace{1mm}

\vspace{2.5mm}

Consider the quadratic rational function $R_\nu$ constructed in theorem \ref{mod blaschke}. Denote the Fatou set and the Julia set for $R_\nu$ by $F(R_\nu)$ and $J(R_\nu)$ respectively. A connected component of $F(R_\nu)$ is called a \ebf{bubble}.

The critical points for $R_\nu$ are $\infty$ and $1$. $\{\infty, R_\nu(\infty)\}$ is a superattracting 2-periodic orbit, and thus is contained in $F(R_\nu)$. Let $\mathcal{B}_\infty$ be the bubble containing $\infty$. The set
\begin{displaymath}
\mathcal{B} := \bigcup_{n=0}^\infty R_\nu^{-n}(\mathcal{B}_\infty)
\end{displaymath}
is the basin of attraction for $\{\infty, R_\nu(\infty)\}$.

The critical point $1$ is contained in the boundary of the Siegel disc $\mathcal{S}_0$. Consider the set of iterated preimages of $\mathcal{S}_0$
\begin{displaymath}
\mathcal{S} := \bigcup_{n=0}^\infty R_\nu^{-n}(\mathcal{S}_0),
\end{displaymath}
It is easy to see that $F(R_\nu) = \mathcal{B} \cup \mathcal{S}$. 

\begin{prop} \label{loc con bubb}
Suppose $U \subset F(R_\nu)$ is a bubble. Then $\partial U$ is locally connected.
\end{prop}

\begin{proof}
The result follows immediately from proposition \ref{basilica unif} and theorem A.
\end{proof}

\begin{prop}
There exists a unique repelling fixed point $\beta$ contained in $\partial \mathcal{B}_\infty$.
\end{prop}

\begin{prop} \label{candidate attach}
Let $u$ be an iterated preimage of $\beta$ (resp. of $1$). Then there are exactly two bubbles $U_1$ and $U_2$ in $\mathcal{B}$ (resp. $\mathcal{S}$) which contain $u$ in their closure.
\end{prop}

A \ebf{bubble ray} for $R_\nu$ can be constructed using bubbles in either $\mathcal{B}$ or $\mathcal{S}$. In the former case, the bubble ray is denoted $\mathcal{R}^\mathcal{B}$, and in the latter case, it is denoted $\mathcal{R}^\mathcal{S}$. The details of the construction will be omitted as it is very similar to the construction of a bubble ray $\mathcal{R}^\mathbf{B}$ for $f_\mathbf{B}$ or $\mathcal{R}^\mathbf{S}$ for $f_\mathbf{S}$.

The \ebf{bubble address of a bubble $U \subset F(R_\nu)$} for $R_\nu$ is defined in the same way as its counterpart for $f_\mathbf{B}$ or $f_\mathbf{S}$. However, since $R_\nu$ is not a polynomial, the external angle of a bubble ray $\mathcal{R}^\mathcal{B}$ or $\mathcal{R}^\mathcal{S}$ cannot be defined using external rays. To circumvent this problem, we need the following theorem.

\begin{thm} \label{interior maps}
There exists a unique conformal map $\Phi_\mathbf{B} : \mathbf{B} \to \mathcal{B}$ such that the bubble addresses are preserved, and the following diagram commutes:
\comdia{\mathbf{B}}{\mathbf{B}}{\mathcal{B}}{\mathcal{B}}{f_{\mathbf{B}}}{R_\nu}{\Phi_{\mathbf{B}}}{\Phi_{\mathbf{B}}}
Likewise, there exists a unique conformal map $\Phi_{\mathbf{S}} : \mathbf{S} \to \mathcal{S}$ such that the bubble addresses are preserved, and the following diagram commutes:
\comdia{\mathbf{S}}{\mathbf{S}}{\mathcal{S}}{\mathcal{S}}{f_{\mathbf{S}}}{R_\nu}{\Phi_{\mathcal{S}}}{\Phi_{\mathcal{S}}}
Furthermore, if $B \subset \mathbf{B}$ (resp. $S \subset \mathbf{S}$) is a bubble, then $\Phi_\mathbf{B}$ (resp. $\Phi_\mathbf{S}$) extends to a homeomorphism between $\overline{B}$ and $\overline{\Phi_\mathbf{B}(B)}$ (resp. $\overline{S}$ and $\overline{\Phi_\mathbf{S}(S)}$).
\end{thm}

\begin{proof}
For each bubble $B \subset \mathbf{B}$, there exists a unique bubble $B' \subset \mathcal{B}$ such that
\begin{displaymath}
\text{adr}(B) = \text{ adr}(B').
\end{displaymath}
Define $\Phi_\mathbf{B}|_B$ to be the unique conformal map between $B$ and $B'$ which sends the root of $B$ to the root of $B'$. Then by construction, $\Phi_\mathbf{B}$ conjugates the dynamics of $f_\mathbf{B}$ and $R_\nu$ restricted to $\mathbf{B}$ and $\mathcal{B}$ respectively. Moreover, $\Phi_\mathbf{B}$ extends continuously to boundary of bubbles by proposition \ref{loc con bubb}.

The map $\Phi_\mathbf{S}$ is similarly defined.
\end{proof}

Let $\mathcal{R}^\mathcal{B} \sim \{B_i\}_{i=0}^\infty$ (resp. $\mathcal{R}^\mathcal{S} \sim \{S_i\}_{i=0}^\infty$) be an infinite bubble ray for $R_\nu$. The \ebf{external angle of $\mathcal{R}^\mathcal{B}$} (resp. \ebf{of $\mathcal{R}^\mathcal{S}$}) is defined to be the external angle of the infinite bubble ray $\mathcal{R}^\mathbf{B} \sim \{\Phi_\mathbf{B}^{-1}(B_i)\}_{i=0}^\infty$ (resp. $\mathcal{R}^\mathbf{S} \sim \{\Phi_\mathbf{S}^{-1}(S_i)\}_{i=0}^\infty$) for $f_\mathbf{B}$ (resp. $f_\mathbf{S}$).  Henceforth, the infinite bubble rays for $R_\nu$ with external angle $t$ will be denoted $\mathcal{R}^\mathcal{B}_t$ and $\mathcal{R}^\mathcal{S}_t$.

\begin{figure}[h]
\centering
\includegraphics[scale=0.6]{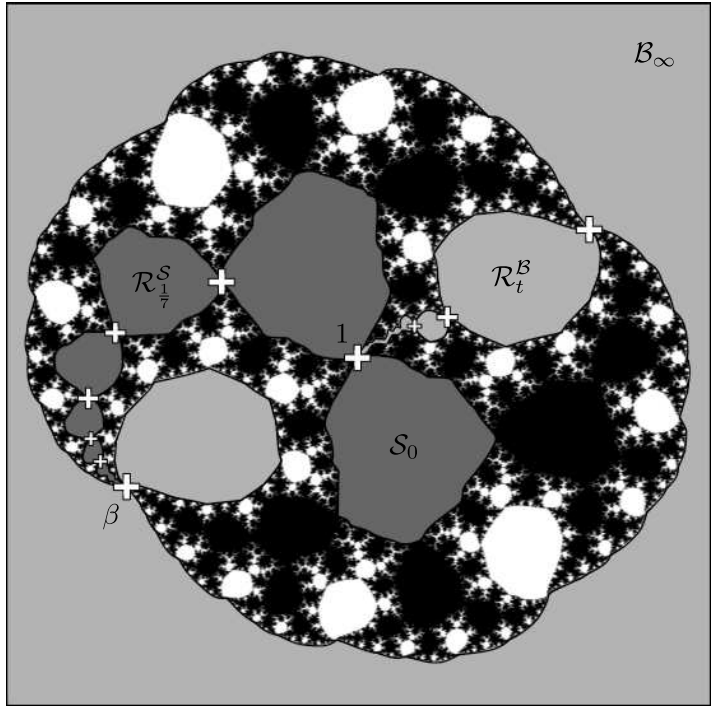}
\caption{The infinite bubble rays $\mathcal{R}^\mathcal{B}_t$, $t \approx -0.143$, and $\mathcal{R}^\mathcal{S}_{\frac{1}{7}}$ for $R_\nu$.}
\label{fig:matedextray}
\end{figure}

\newpage

\section{The Construction of Puzzle Partitions} \label{puzzle part}

\subsection{For the basilica polynomial} \label{basilica puzzle} \hspace{1mm}

\vspace{2.5mm}

Consider the basilica polynomial $f_\mathbf{B}$ discussed in section \ref{basilica bubble}. By definition, the bubble ray $\mathcal{R}^\mathbf{B}_0$ and the external ray $\mathcal{R}^\infty_0$ both land at the same repelling fixed point $\mathbf{k}_\mathbf{B} \in \mathbb{C}$. The \ebf{puzzle partition of level $n$} for $f_\mathbf{B}$ is defined as
\begin{displaymath}
\mathcal{P}^\mathbf{B}_n := \overline{f_\mathbf{B}^{-n}(\mathcal{R}^\mathbf{B}_0 \cup \mathcal{R}^\infty_0)}.
\end{displaymath}
Note that the puzzle partitions form a nested sequence: $\mathcal{P}^\mathbf{B}_0 \subsetneq \mathcal{P}^\mathbf{B}_1 \subsetneq \mathcal{P}^\mathbf{B}_2 \ldots{}$. A \ebf{puzzle piece of level $n$} for $f_\mathbf{B}$ is the closure of a connected component of $\mathbb{C} \setminus \mathcal{P}^\mathbf{B}_n$. By construction, a puzzle piece of level $n$ is mapped homeomorphically onto a puzzle piece of level $n-1$ by $f_\mathbf{B}$.

Let $P^\mathbf{B}$ be a puzzle piece of level $n$. Then $P^\mathbf{B}$ is bounded by two bubble rays $\mathcal{R}^\mathbf{B}_{t_1}$ and $\mathcal{R}^\mathbf{B}_{t_2}$, and two external rays $\mathcal{R}^\infty_{-t_1}$ and $\mathcal{R}^\infty_{-t_2}$, where $t_1 = \frac{i}{2^n}$ and $t_2 = \frac{i+1}{2^n}$ for some $i \in \{0, 1, \ldots{}, 2^n-1\}$. The closed interval $[\frac{i}{2^n}, \frac{i+1}{2^n}] \subset \mathbb{R} / \mathbb{Z}$ is referred to as the \ebf{angular span of $P^\mathbf{B}$}. Henceforth, the puzzle piece for $f_\mathbf{B}$ with angular span $[t_1, t_2]$ will be denoted $P^\mathbf{B}_{[t_1, t_2]}$.

\begin{figure}[h]
\centering
\includegraphics[scale=0.7]{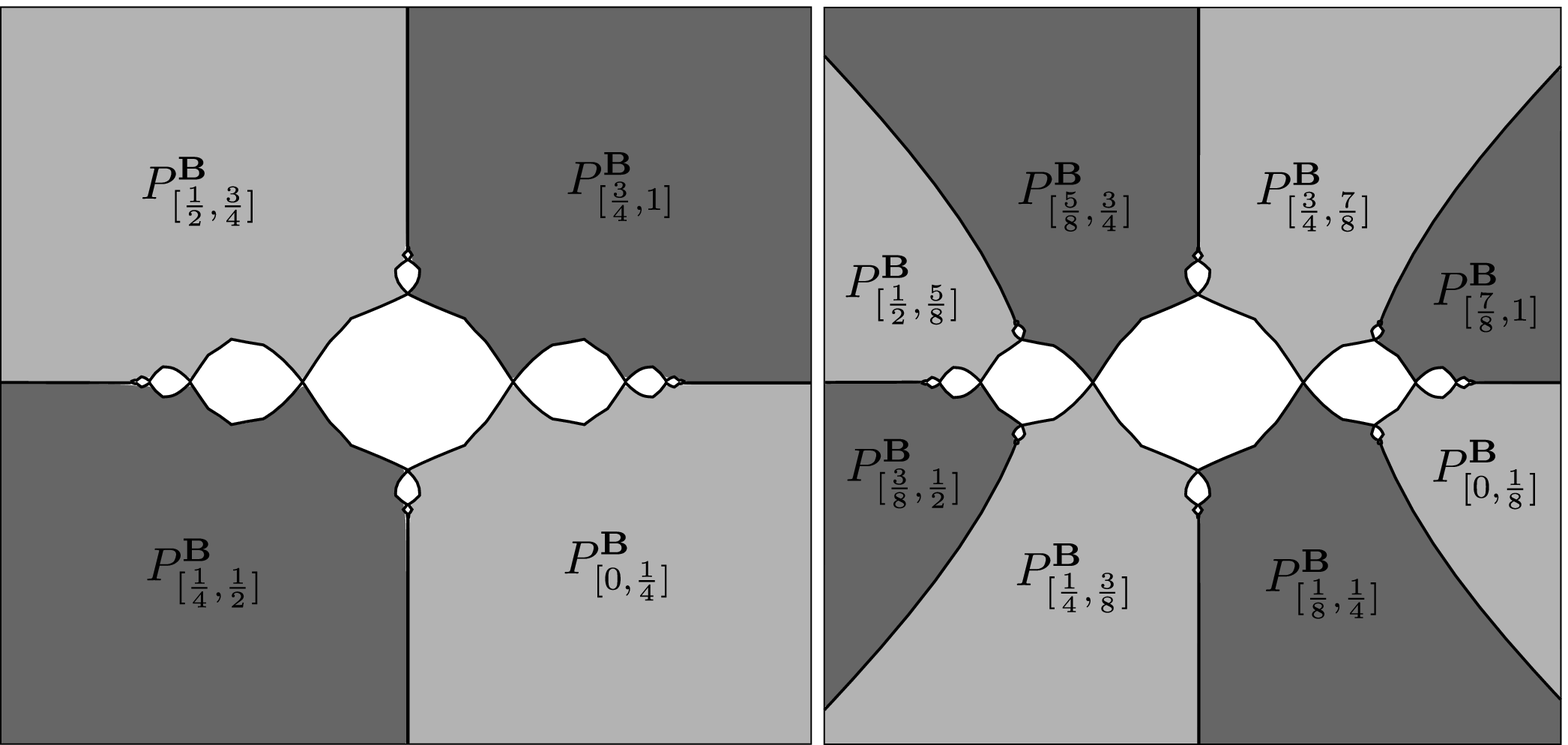}
\caption{The puzzle partition of level $2$ and $3$ for $f_\mathbf{B}$.}
\label{fig:basilicapuzzle}
\end{figure}

\begin{prop} \label{ext ray puzzle seq}
Let $P^\mathbf{B}_{[t_1, t_2]}$ be a puzzle piece. If $t \in [t_1, t_2]$, then $\mathcal{R}^\infty_{-t} \subset P^\mathbf{B}_{[t_1, t_2]}$.
\end{prop}

A \ebf{nested puzzle sequence at $x$} is a collection of puzzle pieces
\begin{displaymath}
\Pi^\mathbf{B} = \{P^\mathbf{B}_{[s_k, t_k]}\}_{k=0}^\infty
\end{displaymath}
such that for all $k \geq 0$, $P^\mathbf{B}_{[s_{k+1}, t_{k+1}]} \subset P^\mathbf{B}_{[s_k, t_k]}$. Note that this is equivalent to the condition that $[s_{k+1}, t_{k+1}] \subset [s_k, t_k]$. The set
\begin{displaymath}
L(\Pi^\mathbf{B}) := \bigcap_{k=0}^\infty P^\mathbf{B}_{[s_k, t_k]}
\end{displaymath}
is called the \ebf{limit of $\Pi^\mathbf{B}$}.

The \ebf{external angle $t \in \mathbb{R} / \mathbb{Z}$ of $\Pi^\mathbf{B}$} is defined by
\begin{displaymath}
\{t\} = \bigcap_{k=0}^\infty [s_k, t_k].
\end{displaymath}
Henceforth, a nested puzzle sequence for $f_\mathbf{B}$ with external angle $t \in \mathbb{R} / \mathbb{Z}$ will be denoted by $\Pi^\mathbf{B}_t$.

A nested puzzle sequence $\Pi^\mathbf{B}_t$ is said to be \ebf{maximal} if there is no nested puzzle sequence which contains $\Pi^\mathbf{B}_t$ as a proper subset. If two nested puzzle sequences are contained in the same maximal nested puzzle sequence, they are said to be \ebf{equivalent}.

\begin{prop}
Suppose $\Pi^\mathbf{B}_s$ and $\Pi^\mathbf{B}_t$ are two equivalent nested puzzle sequences. Then $s=t$, and $L(\Pi^\mathbf{E}_s) = L(\Pi^\mathbf{E}_t)$. 
\end{prop}

\begin{proof}
Let $\Pi^\mathbf{B}_s = \{P^\mathbf{B}_{[s_k, t_k]}\}_{k=0}^\infty$, and let $\hat{\Pi}^\mathbf{B}_u = \{P^\mathbf{B}_{[r_k, u_k]}\}_{k=0}^\infty$ be the maximal nested puzzle sequence containing $\Pi^\mathbf{B}_s$. Since $P^\mathbf{B}_{[s_k, t_k]} \subseteq P^\mathbf{B}_{[r_k, u_k]}$ for all $k \in \mathbb{N}$, we have
\begin{displaymath}
L(\Pi^\mathbf{B}_s) \subset L(\Pi^\mathbf{B}_u).
\end{displaymath}
On the other hand, since $\Pi^\mathbf{B}_s \subset \hat{\Pi}^\mathbf{B}_u$, we have
\begin{displaymath}
L(\hat{\Pi}^\mathbf{B}_u) \subset L(\Pi^\mathbf{B}_s).
\end{displaymath}
The proof that $s=t$ is similar.
\end{proof}

\begin{prop} \label{shrink to ray}
Let $\Pi^\mathbf{B}_t := \{P^\mathbf{B}_{[s_k, t_k]}\}_{k=0}^\infty$ be a nested puzzle sequence. Then
\begin{displaymath}
L(\Pi^\mathbf{B}_t) = \mathcal{R}^\infty_{-t} \cup \{x\},
\end{displaymath}
where $x \in J_\mathbf{B}$ is the landing point of $\mathcal{R}^\infty_{-t}$.
\end{prop}

\begin{proof}
Observe that for each $k$, we have
\begin{displaymath}
\mathbf{A}^\infty_\mathbf{B} \cap P^\mathbf{B}_{[s_k, t_k]} = \bigcup_{s \in [s_k, t_k]} \mathcal{R}^\infty_{-s}. 
\end{displaymath}
Since
\begin{displaymath}
J_\mathbf{B} \cap P^\mathbf{B}_{[s_k, t_k]} = K_\mathbf{B} \cap \partial (P^\mathbf{B}_{[s_k, t_k]} \cap \mathbf{A}^\infty_\mathbf{B}),
\end{displaymath}
we see that $J_\mathbf{B} \cap P^\mathbf{B}_{[s_k, t_k]}$ consists of landing points of $\mathcal{R}^\infty_{-s}$, $s \in [s_k, t_k]$.

If $s \neq t$, then for $k$ sufficiently large, $s \not\in [s_k, t_k]$, which means the landing point of $\mathcal{R}^\infty_{-s}$ is not included in $L(\Pi^\mathbf{B}_t)$. The result follows.
\end{proof}

\begin{prop} \label{proof puzzle count}
Let $x \in J_\mathbf{B}$. If $x$ is an iterated preimage of $\mathbf{k}_\mathbf{B}$ or $\mathbf{b}$, then for all sufficiently large $n$, $x$ is contained in exactly two puzzle pieces of level $n$. Otherwise, $x$ is contained in a unique puzzle piece of level $n$.
\end{prop}

\begin{proof}
We consider the following four cases:
\begin{enumerate}[label=\roman*)]
\item $x$ is an iterated preimage of $\mathbf{b}$.
\item There exists a unique bubble $B \subset \mathbf{B}$ such that $x \in \partial B$.
\item $x$ is an iterated preimage of $\mathbf{k}_\mathbf{B}$.
\item Otherwise.
\end{enumerate}

\vspace{2.5mm}

\noindent \emph{Case i)} Suppose that $x$ is an iterated preimage of $\mathbf{b}$. By proposition \ref{parentchild}, there exist exactly two bubbles $B_1$ and $B_2$ which contain $x$ in their boundary. Moreover, we have $\{x\} = \overline{B_1} \cap \overline{B_2}$. Note that $B_1$ and $B_2$ are eventually mapped to $\mathbf{B}_0 \subset \mathcal{R}^\mathbf{B}_0$ under $f_\mathbf{B}$. Hence, there exists $m \in \mathbb{N}$ such that for all $n > m$, $B_1 \cup B_2 \subset \mathcal{P}^\mathbf{B}_n$.

Let $n > m$. $\mathcal{P}^\mathbf{B}_n$ contains finitely many bubble rays whose landing points are all distinct from $x$. Thus, we can choose a sufficiently small disc $D$ centered at $x$ such that $D \cap \mathcal{P}^\mathbf{B}_n \subset (\overline{B_1} \cup \overline{B_2})$. This implies that $D \cap (\mathbb{C} \setminus \mathcal{P}^\mathbf{B}_n)$ has two connected components. Observe that every puzzle piece of level $n$ that contains $x$ must contain exactly one of the components of $D \cap (\mathbb{C} \setminus \mathcal{P}^\mathbf{B}_n)$. The result follows.

\vspace{2.5mm}

\noindent \emph{Case ii)} The proof is completely analogous to Case i).

\vspace{2.5mm}

\noindent \emph{Case iii)} Suppose $x$ is an iterated preimage of $\mathbf{k}_\mathbf{B}$. Note that $\mathcal{R}^\mathbf{B}_0$ is the only bubble ray which lands on $\mathbf{k}_\mathbf{B}$. Hence, $x$ is the landing point of a single bubble ray $\mathcal{R}^\mathbf{B}_t$ where $t = \frac{i}{2^m}$ for some $m \in \mathbb{N}$ and $i \in \{0, 1, \ldots{}, 2^m-1\}$.

Observe that for all $n \geq m$, $x$ is contained in the puzzle partition $\mathcal{P}^\mathbf{B}_n$ of level $n$. This implies that if $x$ is contained in some puzzle piece $P^\mathbf{B}_{[t_1, t_2]}$ of level $n$, then $x$ is contained in its boundary. The boundary of $P^\mathbf{B}_{[t_1, t_2]}$ is a subset of $\mathcal{R}^\mathbf{B}_{t_1} \cup \mathcal{R}^\infty_{t_1} \cup \mathcal{R}^\mathbf{B}_{t_2} \cup \mathcal{R}^\infty_{t_2}$. It follows from our previous remark that $t_1 = \frac{i}{2^m}$ or $t_2 = \frac{i}{2^m}$. Therefore, $P^\mathbf{B}_{[\frac{2^{n-m}i-1}{2^n}, \frac{i}{2^m}]}$ and $P^\mathbf{B}_{[\frac{i}{2^m}, \frac{2^{n-m}i+1}{2^n}]}$ are the only two puzzle pieces of level $n$ which contain $x$.

\vspace{2.5mm}

\noindent \emph{Case iv)} If $x$ is not contained in the boundary of any bubble, and $x$ is not an iterated preimage of $\mathbf{k}_\mathbf{B}$, then $x$ is disjoint from every puzzle partition. Hence, $x$ must be contained in a unique component of its complement.
\end{proof}

\begin{cor} \label{proof seq count}
Let $x \in J_\mathbf{B}$. If $x$ is an iterated preimage of $\mathbf{k}_\mathbf{B}$ or $\mathbf{b}$, then there are exactly two maximal nested puzzle sequences whose limit is equal to $\{x\}$. Otherwise, there is a unique maximal nested puzzle sequence whose limit is equal to $\{x\}$.
\end{cor}

\begin{prop} \label{B access}
Let $x \in J_\mathbf{B}$. If $x$ is an iterated preimage of $\mathbf{b}$, then $x$ is biaccessible. Otherwise, $x$ is uniaccessible.
\end{prop}

\begin{proof}
This follows immediately from proposition \ref{shrink to ray} and corollary \ref{proof seq count}.
\end{proof}

\subsection{For the Siegel polynomial} \label{siegel puzzle} \hspace{1mm}

\vspace{2.5mm}

Consider the Siegel polynomial $f_\mathbf{S}$ discussed in section \ref{siegel bubble}. By definition, the bubble ray $\mathcal{R}^\mathbf{B}_0$ and the external ray $\mathcal{R}^\infty_0$ both land at the same point $\mathbf{k}_\mathbf{S} \in \mathbb{C}$. A \ebf{puzzle partition $\mathcal{P}^\mathbf{S}_n$}, a \ebf{puzzle piece $P^\mathbf{S}_{[t_1, t_2]}$}, and a \ebf{nested puzzle sequence $\Pi^\mathbf{S}_t$} for $f_\mathbf{S}$ are defined in the same way as their counterparts for $f_\mathbf{B}$.

\begin{figure}[h]
\centering
\includegraphics[scale=0.7]{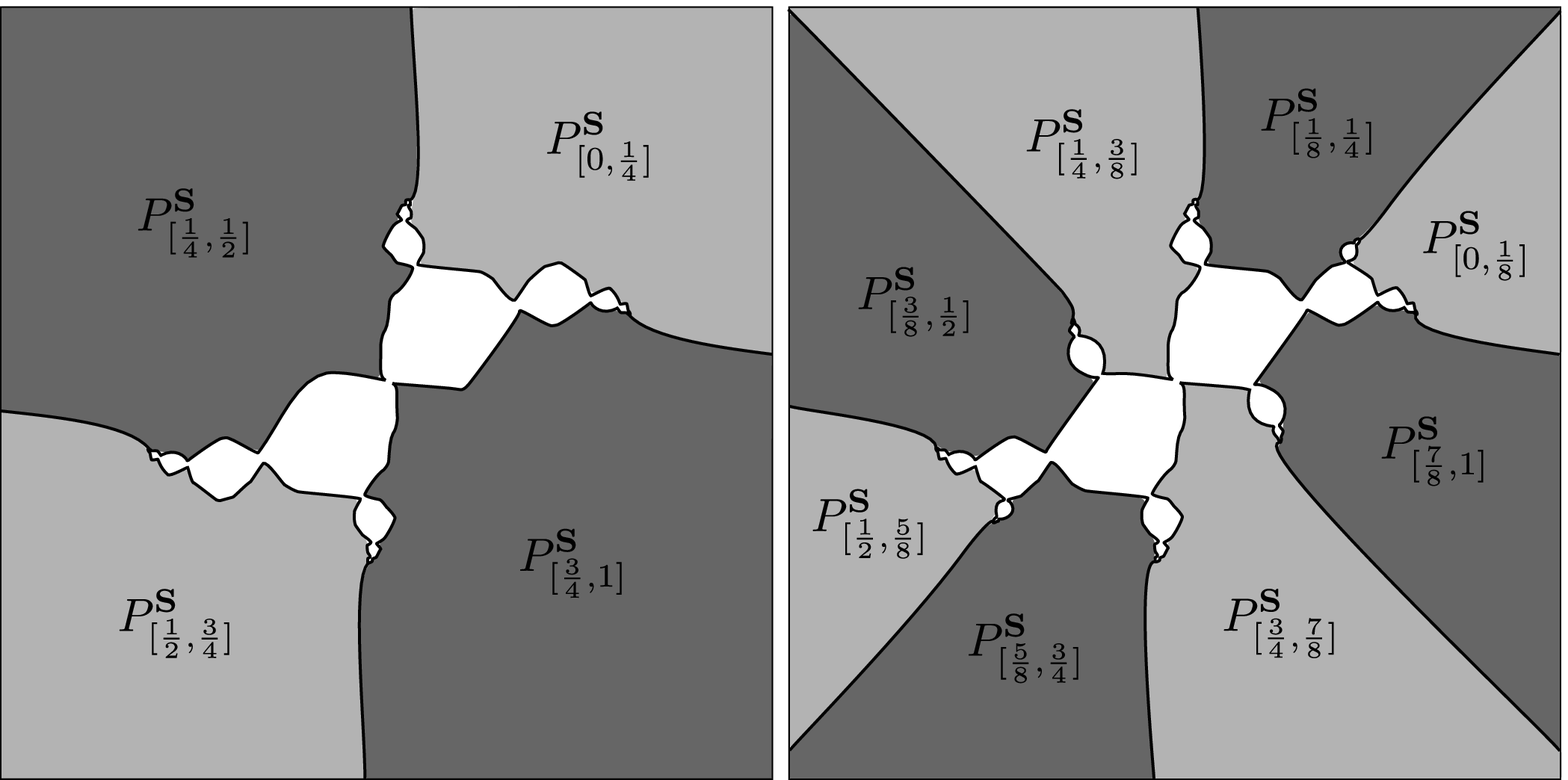}
\caption{The puzzle partition of level $2$ and $3$ for $f_\mathbf{S}$.}
\label{fig:siegelpuzzle}
\end{figure}

The following two results are analogs of proposition \ref{shrink to ray} and \ref{proof puzzle count}. The proofs are identical, and hence, they will be omitted here.

\begin{prop} \label{shrink to ray 2}
Let $\Pi^\mathbf{S}_t := \{P^\mathbf{S}_{[s_k, t_k]}\}_{k=0}^\infty$ be a nested puzzle sequence. Then
\begin{displaymath}
L(\Pi^\mathbf{S}_t) = \mathcal{R}^\infty_t \cup \{x\},
\end{displaymath}
where $x \in J_\mathbf{S}$ is the landing point of $\mathcal{R}^\infty_t$.
\end{prop}

\begin{prop}
Let $x \in J_\mathbf{S}$. If $x$ is an iterated preimage of $\mathbf{k}_\mathbf{S}$ or $0$, then for all sufficiently large $n$, $x$ is contained in exactly two puzzle pieces of level $n$. Otherwise, $x$ is contained in a unique puzzle piece of level $n$.
\end{prop}

\begin{cor} \label{num seq siegel}
Let $x \in J_\mathbf{S}$. If $x$ is an iterated preimage of $\mathbf{k}_\mathbf{S}$ or $0$, then there are exactly two maximal nested puzzle sequences whose limit is equal to $\{x\}$. Otherwise, there exists a unique maximal nested puzzle sequence whose limit is equal to $\{x\}$.
\end{cor}

\begin{prop} \label{S access}
Let $x \in J_\mathbf{S}$. If $x$ is an iterated preimage of $0$, then $x$ is biaccessible. Otherwise, $x$ is uniaccessible.
\end{prop}

\begin{proof}
This follows immediately from proposition \ref{shrink to ray 2} and corollary \ref{num seq siegel}.
\end{proof}

\subsection{For the candidate mating} \label{candidate puzzle} \hspace{1mm}

\vspace{2.5mm}

Consider the quadratic rational function $R_\nu$ constructed in theorem \ref{mod blaschke}.

\begin{prop} \label{periodic landing}
Let $\mathcal{R}_t = \mathcal{R}^\mathcal{B}_t$ or $\mathcal{R}^\mathcal{S}_t$ be an infinite bubble ray. If $t$ is rational, then $\mathcal{R}_t$ lands. If $t$ is $p$-periodic, then $\mathcal{R}_t$ lands at a repelling $p$-periodic point.
\end{prop}

\begin{proof}
Let $\Lambda$ be the post critical set for $R_\nu$, and let $\Omega$ be the set of cluster points for $\mathcal{R}_t$. Observe that
\begin{displaymath}
R_\nu^p : \hat{\mathbb{C}} \setminus R_\nu^{-p}(\Lambda \cup \Omega) \to \hat{\mathbb{C}} \setminus (\Lambda \cup \Omega)
\end{displaymath}
is a covering of hyperbolic spaces. Moreover, since $\Omega \cup \Lambda \subsetneq R_\nu^{-p}(\Omega \cup \Lambda)$, the inclusion map
\begin{displaymath}
\iota : \hat{\mathbb{C}} \setminus R_\nu^{-p}(\Lambda \cup \Omega) \to \hat{\mathbb{C}} \setminus (\Lambda \cup \Omega)
\end{displaymath} 
is a strict contraction in the hyperbolic metric. Hence, the map $\iota \circ R_\nu^{-p}$ lifts to the universal cover $\mathbb{D}$ of $\hat{\mathbb{C}} \setminus (\Lambda \cup \Omega)$ to a map
\begin{displaymath}
\hat{R}_\nu^{-p} : \mathbb{D} \to \mathbb{D}
\end{displaymath}
which is also a strict contraction in the hyperbolic metric.

Now, choose a bubble $U \subset \mathcal{R}_t$ such that $\text{gen}(U) > 1$, and let $x_0$ be a point contained in $U$. For every $k \geq 1$, there exists a unique point $x_k \in \mathcal{R}_t$ such that $R_\nu^{kp}(x_k) = x_0$. Let $\gamma_0 \subset \mathcal{R}_t$ be a curve from $x_0$ to $x_1$, and let $\gamma_k$ be the unique component of $R_\nu^{-kp}(\gamma_0)$ whose end points are $x_k$ and $x_{k+1}$.

By the strict contraction property of $\hat{R}_\nu^{p}$, the hyperbolic lengths of $\gamma_n$ must go to zero as $n$ goes to infinity. Hence, if $z \in \Omega$, then for any neighbourhood $N$ of $z$, there exists a smaller neighbourhood $N' \subset N$ such that if $\gamma_n \cap N' \neq \varnothing$, then $\gamma_n \subset N$. In other words, $R_\nu^p(N) \cap N \neq \varnothing$. Since this is true for all neighbourhood of $z$, $z$ must be a fixed point for $R_\nu^p$.

The set of fixed points for $R_\nu^p$ is discrete. Since $\Omega$ is connected, this implies that $\Omega$ must be equal to the single point set $\{z\}$. By Snail lemma (see e.g. [M1]), we conclude that $z$ is a repelling fixed point.

If $t$ is strictly preperiodic, then $\mathcal{R}_t$ is the preimage of some periodic infinite bubble ray. The result follows.
\end{proof}

\begin{prop} \label{base}
The bubble rays $\mathcal{R}^\mathcal{B}_0$ and $\mathcal{R}^\mathcal{S}_0$ for $R_\nu$ both land at the same repelling fixed point $\kappa \in \mathbb{C}$.
\end{prop}

\begin{proof}
The quadratic rational map $R_\nu$ has exactly three fixed points, two of which must be the Siegel fixed point $0$ and the repelling fixed point $\beta$. Clearly, a bubble ray cannot land on $0$, so it suffices to prove that a fixed bubble ray cannot land on $\beta$.

Let $D$ be a sufficiently small disc centered at $\beta$ such that $R_\nu$ is conformal on $D$. The set $D \cap (\hat{\mathbb{C}} \setminus \overline{\mathcal{B}_\infty \cup R_\nu(\mathcal{B}_\infty)})$ has two connected components $D_1$ and $D_2$ such that $D_1 \subset R_\nu(D_2)$ and $D_2 \subset R_\nu(D_1)$. Suppose $\mathcal{R}$ is a bubble ray that lands on $\beta$. Then $\mathcal{R}$ must be disjoint from either $D_1$ or $D_2$. Hence, $\mathcal{R}$ cannot be fixed.
\end{proof}

Define the \ebf{puzzle partition of level $n$} for $R_\nu$ by
\begin{displaymath}
\mathcal{P}_n := \overline{R_\nu^{-n}(\mathcal{R}^\mathcal{B}_0 \cup \mathcal{R}^\mathcal{S}_0)}.
\end{displaymath}
Note that the puzzle partitions form a nested sequence: $\mathcal{P}_0 \subsetneq \mathcal{P}_1 \subsetneq \mathcal{P}_2 \ldots{}$. A \ebf{puzzle piece of level $n$} is the closure of a connected component of $\hat{\mathbb{C}} \setminus \mathcal{P}_n$. By construction, a puzzle piece of level $n$ is mapped homeomorphically onto a puzzle piece of level $n-1$ by $R_\nu$.

Let $P$ be a puzzle piece of level $n$. Then $P$ is bounded by two pairs of bubble rays: $\mathcal{R}^\mathbf{B}_{s_1}$ and $\mathcal{R}^\mathbf{B}_{s_2}$, and $\mathcal{R}^\mathcal{S}_{t_1}$ and $\mathcal{R}^\mathcal{S}_{t_2}$, where $s_1 = \frac{i}{2^n}$, $s_2 = \frac{i+1}{2^n}$, $t_1 = \frac{j}{2^n}$ and $t_2 = \frac{j+1}{2^n}$ for some $i, j \in \{0, 1, \ldots{}, 2^n-1\}$. The closed intervals $[\frac{i}{2^n}, \frac{i+1}{2^n}]$ and $[\frac{j}{2^n}, \frac{j+1}{2^n}]$ are referred to as the \ebf{angular span of $P$ with respect to $\mathcal{B}$} and $\mathcal{S}$ respectively.

\begin{prop} \label{puzzle cor}
Consider the maps $\Phi_\mathbf{B} : \mathbf{B} \to \mathcal{B}$ and $\Phi_\mathbf{S} : \mathbf{S} \to \mathcal{S}$ defined in proposition \ref{interior maps}. Let $P$ be a puzzle piece of level $n$ for $R_\nu$ whose angular span with respect to $\mathcal{B}$ and $\mathcal{S}$ are equal to $[s_1, s_2]$ and $[t_1, t_2]$ respectively. Then $\Phi_\mathbf{B}$ restricts to a map between $\mathbf{B} \cap P^\mathbf{B}_{[s_1, s_2]}$ and $\mathcal{B} \cap P$. Likewise, $\Phi_\mathbf{S}$ restricts to a map between $\mathbf{S} \cap P^\mathbf{S}_{[t_1, t_2]}$ and $\mathcal{S} \cap P$.
\end{prop}

\begin{proof}
Let $B \not\subset \mathcal{P}^\mathbf{B}_n$ be a bubble in $\mathbf{B}$ and let $\mathcal{R}^\mathbf{B} \sim \{B_i\}_{i=0}^m$ be the unique finite bubble ray for $f_\mathbf{B}$ such that $B_m = B$. The corresponding finite bubble ray for $R_\nu$ is $\mathcal{R}^\mathcal{B} \sim \{\Phi_\mathbf{B}(B_i)\}_{i=0}^m$. Suppose $k$ is the largest value of $i$ such that $B_i \subset \mathcal{P}^\mathbf{B}_n$. Since $\Phi_\mathbf{B}$ extends to a homeomorphism between $\overline{B_k}$ and $\overline{\Phi_\mathbf{B}(B_k)}$, it must preserve the cyclic order of the roots of bubbles contained in $\partial B$. Thus, we see that $\Phi_\mathbf{B}(B_{k+1}) \subset P_{[s_1, s_2]}$ if and only if $B_{k+1} \subset P^\mathbf{B}_{[s_1, s_2]}$. This readily implies that $\bigcup_{i=k+1}^m \Phi_\mathbf{B}(B_i) \subset P_{[s_1, s_2]}$ if and only if $\bigcup_{i=k+1}^m B_i \subset P^\mathbf{B}_{[s_1,s_2]}$.

The proof is completely analogous for bubbles in $\mathbf{S}$.
\end{proof}

\begin{cor} \label{ray in puzzle}
Let $P$ be a puzzle piece of level $n$, whose angular span with respect to $\mathcal{B}$ and $\mathbf{S}$ is equal to $[s_1, s_2]$ and $[t_1, t_2]$ respectively. If $s \in [s_1, s_2]$, then the accumulation set of $\mathcal{R}^\mathcal{B}_s$ is contained in $P$. Likewise, if $t \in [t_1, t_2]$, then the accumulation set of $\mathcal{R}^\mathbf{S}_t$ is contained in $P$.
\end{cor}

\begin{prop} \label{angular span}
Let $P$ be a puzzle piece of level $n$ for $R_\nu$ whose angular span with respect to $\mathcal{B}$ and $\mathcal{S}$ are equal to $[s_1, s_2]$ and $[t_1, t_2]$ respectively. Then $[s_1, s_2] = [t_1, t_2]$.
\end{prop}

\begin{proof}
We proceed by induction on $n$. The case $n=0$ follows from proposition \ref{base}. Assume that the statement is true for $n$. We need to check that it is also true for $n+1$.

We have $s_1 = \frac{i}{2^n}$ and $s_2 = \frac{i+1}{2^n}$ for some $i \in \{0, 1, \ldots{}, 2^n-1\}$. Let $s = \frac{2i+1}{2^{n+1}}$. Consider the puzzle pieces $P^\mathbf{B}_{[s_1, s_2]}$ and $P^\mathbf{S}_{[s_1, s_2]}$ for $f_\mathbf{B}$ and $f_\mathbf{S}$ respectively. Among all the bubble rays contained in the puzzle partition $\mathcal{P}^\mathbf{B}_{n+1}$ for $f_\mathbf{B}$, only $\mathcal{R}^\mathbf{B}_s$ lands in the interior of $P^\mathbf{B}_{[s_1, s_2]}$. Similarly, among all the bubble rays contained in the puzzle partition $\mathcal{P}^\mathbf{S}_{n+1}$ for $f_\mathbf{S}$, only $\mathcal{R}^\mathbf{S}_s$ lands in the interior of $P^\mathbf{S}_{[s_1, s_2]}$. It follows from lemma \ref{puzzle cor} that $\mathcal{R}^\mathcal{B}_s = \Phi_\mathbf{B}^{-1}(\mathcal{R}^\mathbf{B}_s)$ and $\mathcal{R}^{\mathcal{S}}_s = \Phi_\mathbf{S}^{-1}(\mathcal{R}^\mathbf{S}_s)$ are the only two bubble rays contained in the puzzle partition $\mathcal{P}_{n+1}$ for $R_\nu$ that land in the interior of $P$. This implies that $\mathcal{R}^\mathcal{B}_s$ and $\mathcal{R}^\mathcal{S}_s$ must land at the same point. It is not difficult to see from this that the claim must be true for puzzle pieces of level $n+1$.
\end{proof}

By virtue of proposition \ref{angular span}, the angular span of a puzzle piece $P$ with respect to $\mathcal{B}$ or $\mathcal{S}$ will henceforth be referred to as simply the \ebf{angular span of $P$}. Furthermore, a puzzle piece for $R_\nu$ with angular span $[t_1, t_2]$ will be denoted $P_{[t_1, t_2]}$.

\begin{figure}[h]
\centering
\includegraphics[scale=0.7]{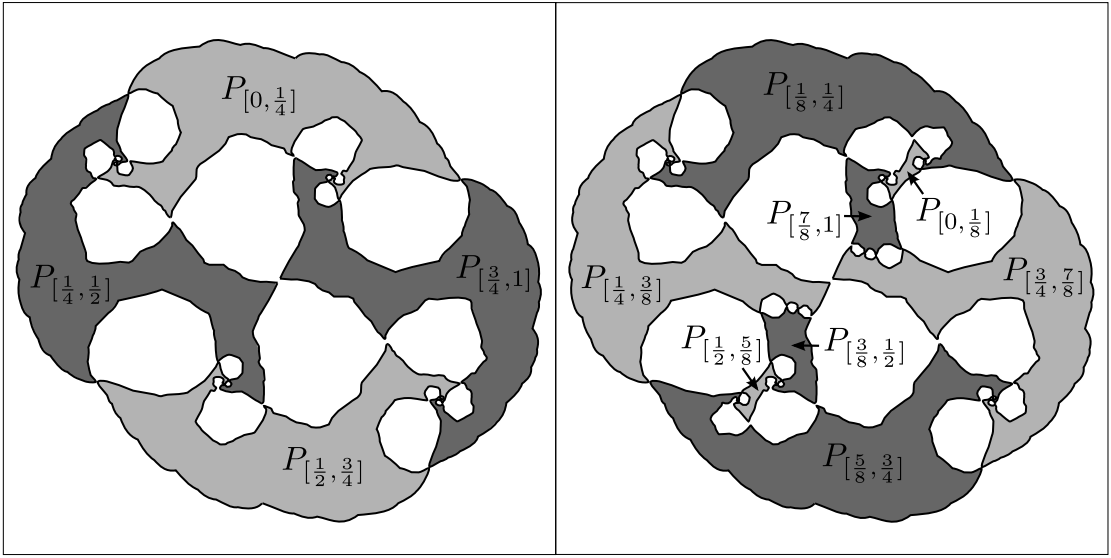}
\caption{The puzzle partition of level $2$ and $3$ for $R_\nu$.}
\label{fig:matedpuzzles}
\end{figure}

A \ebf{nested puzzle sequence $\Pi_t$} for $R_\nu$ is defined in the same way as its counterpart for $f_\mathbf{B}$. We say that \ebf{$\Pi_t$ shrinks to $x$} if its limit $L(\Pi_t)$ is equal to $\{x\}$.

\begin{prop} \label{maximal shrink}
Let $\Pi_t = \{P_{[s_k, t_k]}\}_{k=0}^\infty$ be a nested puzzle sequence, and let $\hat{\Pi}_t = \{P_{[r_k, u_k]}\}_{k=0}^\infty$ be the unique maximal nested puzzle sequence containing $\Pi_t$. Then $\Pi_t$ shrinks to a point $x \in J(R_\nu)$ if and only if $\hat{\Pi}_t$ does.
\end{prop}

The following result can be proved the same way as proposition \ref{proof puzzle count}.

\begin{prop} \label{num puzzle}
Let $x \in J(R_\nu)$. If $x$ is an iterated preimage of $\kappa$, $\beta$ or $1$, then for all sufficiently large $n$, $x$ is contained in exactly two puzzle pieces of level $n$. Otherwise, $x$ is contained in a unique puzzle piece of level $n$.
\end{prop}

\begin{cor} \label{num sequence}
Let $x \in J(R_\nu)$. If $x$ is an iterated preimage of $\kappa$, $\beta$ or $1$, then there are exactly two maximal nested puzzle sequences whose limit contains $x$. Otherwise, there is exactly one maximal nested puzzle sequence whose limit contains $x$.
\end{cor}

\begin{prop} \label{no connection}
Suppose $X \subset \hat{\mathbb{C}} \setminus (\mathcal{B}_\infty \cup R_\nu(\mathcal{B}_\infty) \cup \mathcal{S}_0)$ is a non-recurring closed set (that is, for all $n \in \mathbb{N}$, $R_\nu^n(X) \cap X = \varnothing$). Then the set $\mathcal{B}_\infty \cup R_\nu(\mathcal{B}_\infty) \cup \mathcal{S}_0 \cup X$ is disconnected.
\end{prop}

\begin{proof}
Suppose towards a contradiction that $\mathcal{B}_\infty \cup R_\nu(\mathcal{B}_\infty) \cup \mathcal{S}_0 \cup X$ is connected. Without loss of generality, we may assume that $\mathcal{B}_\infty \cup \mathcal{S}_0 \cup X$ is connected.

Observe that $\hat{\mathbb{C}} \setminus (\mathcal{B}_\infty \cup \mathcal{S}_0 \cup X \cup R_\nu^2(X))$ is disconnected, and that at least one of its components intersects $\partial \mathcal{S}_0$ but does not intersect $\partial R_\nu(\mathcal{B}_\infty)$. Denote this component by $P$.

Since $X$ is non-recurring, observe that $R_\nu^{2n+1}(X) \cap P = \varnothing$ for all $n \geq 0$. Choose a point $x_1 \in \overline{\mathcal{S}_0} \cap R_\nu(X)$. Since the orbit of $x_1$ under $R_\nu^2$ is dense in $\partial \mathcal{S}_0$, there exists $N \geq 0$ such that $R_\nu^{2N+1}(x_1) \in \partial \mathcal{S}_0 \cap P$. This is a contradiction.
\end{proof}

\begin{prop} \label{disjointness for rnu}
Let $\Pi_t = \{P_{[s_k, t_k]}\}_{k=0}^\infty$ be a nested puzzle sequence for $R_\nu$. Its limit $L(\Pi_t)$ cannot intersect the boundary of bubbles from both $\mathcal{B}$ and $\mathcal{S}$.
\end{prop}

\begin{proof}
Suppose that $L(\Pi_t)$ intersects the boundary of bubbles from both $\mathcal{B}$ and $\mathcal{S}$. By considering its image under a large enough iterate of $R_\nu$, we may assume that $L(\Pi_t)$ intersects the boundary of $\mathcal{B}_\infty$ and $\mathcal{S}_0$.

Observe that the limit set of any nested puzzle sequence is either pre-periodic or non-recurrent. Since $L(\Pi_t)$ contains a point in $\partial \mathcal{S}_0$, it must be non-recurrent. It is also easy to see that $L(\Pi_t)$ must be closed, connected and contained in $\hat{\mathbb{C}} \setminus (\mathcal{B}_\infty \cup R_\nu(\mathcal{B}_\infty) \cup \mathcal{S}_0)$. This contradicts proposition \ref{no connection}.
\end{proof}

The following result is proved in the next two sections.

\begin{thm}[the Shrinking Theorem] \label{total shrinkage}
Every nested puzzle sequence for $R_\nu$ shrinks to a point.
\end{thm}

\newpage

\section{A Priori Bounds for Critical Circle Maps} \label{secapriori}

A $C^2$ homeomorphism $f : S^1 \to S^1$ is called a \ebf{critical circle map} if it has a unique critical point $c \in S^1$ of cubic type. Let $\rho = \rho(f)$ be the rotation number of $f$. In this section, $f$ will be analytic, and $\rho$ will be irrational.

The rotation number $\rho$ can be represented as an infinite continued fraction:
\begin{displaymath}
\rho = [a_1 : a_2 : a_3 : \ldots{}] = \frac{1}{a_1+\frac{1}{a_2+\frac{1}{a_3 + \ldots{}}}}.
\end{displaymath}
The \ebf{$n$th partial convergent of $\rho$} is the rational number
\begin{displaymath}
\frac{p_n}{q_n} = [a_1: \ldots : a_n].
\end{displaymath}
The sequence of denominators $\{q_n\}_{n=1}^\infty$ represent the \ebf{closest return times} of the orbit of any point to itself. It satisfies the following inductive relation:
\begin{displaymath}
q_{n+1} = a_n q_n + q_{n-1}.
\end{displaymath}

Let $\Delta_n \subset S^1$ be the closed arc containing $c$ with end points at $f^{q_n}(c)$ and $f^{q_{n+1}}(c)$. $\Delta_n$ can be expressed as the union of two closed arcs $A_n$ and $A_{n+1}$, where $A_n$ is the closed arc with end points at $c$ and $f^{q_n}(c)$. $A_n$ is called the \ebf{$n$th critical arc}. The $q_n$th iterated preimage of $A_n$ under $f$ is denoted by $A_{-n}$. The set of closed arcs
\begin{displaymath}
\mathcal{P}^{S^1}_n = \{A_n, f(A_n), \ldots, f^{q_{n+1}-1}(A_n)\} \cup \{A_{n+1}, f(A_{n+1}), \ldots, f^{q_n-1}(A_{n+1})\},
\end{displaymath}
which are disjoint except at the end points, is a partition of $S^1$. $\mathcal{P}^{S^1}_n$ is called the \ebf{dynamical partition of level $n$}. The following is an important estimate regarding dynamical partitions due to Swi\c{a}tek and Herman (see \cite{Sw}):

\begin{thm} [Real \emph{a priori} bounds]\label{real bound}
Let $f : S^1 \to S^1$ be a critical circle map with an irrational rotation number $\rho$. Then for all $n$ sufficiently large, every pair of adjacent atoms in $\mathcal{P}^{S^1}_n$ have $K$-commensurate diameters for some universal constant $K > 1$.
\end{thm}

Below, we present an adaptation of complex a priori bounds of \cite{Y1} (see also \cite{YZ}) to our setting.

Consider the quadratic rational function $R_\nu$ discussed in section \ref{candidate bubble} and \ref{puzzle part}. Denote the Siegel disc for $R_\nu$ by $\mathcal{S}_0$. By theorem \ref{mod blaschke}, there exist a Blaschke product $F_\nu$ and a quasiconformal map $\phi : \hat{\mathbb{C}} \setminus \mathbb{D} \to \hat{\mathbb{C}} \setminus \mathcal{S}_0$ such that
\begin{displaymath}
R_\nu(z) = \phi \circ F_\nu \circ \phi^{-1}(z)
\end{displaymath}
for all $z \in \hat{\mathbb{C}} \setminus \mathcal{S}_0$.

Since $\{\infty, R_\nu(\infty)\}$ is a superattracting 2-periodic orbit for $R_\nu$, $\{\infty, F_\nu(\infty)\}$ and $\{0, F_\nu(0)\}$ are superattracting 2-periodic orbits for $F_\nu$. Denote the bubble (connected component of the Fatou set) for $F_\nu$ containing $0$ and $\infty$ by $\mathcal{A}_0$ and $\mathcal{A}_\infty$ respectively. Note that by theorem \ref{fnu}, the restriction of $F_\nu$ to $S^1$ is a critical circle map.

A \ebf{puzzle piece of level $n$ for $F_\nu$} is the image of a puzzle piece of level $n$ for $R_\nu$ under $\phi^{-1}$. The \ebf{$n$th critical puzzle piece}, denoted $P^{crit}_n$, is defined inductively as follows:
\begin{enumerate}[label=(\roman{*})]
\item $P^{crit}_0$ is the puzzle piece of level $1$ which contains the first critical arc $A_1$.
\item $P^{crit}_n$ is the puzzle piece which contains the preimage arc $A_{-n}$, and is mapped homeomorphically onto $P^{crit}_{n-1}$ by $F_\nu^{q_n}$.
\end{enumerate}
Observe that $\Pi_{\text{even}} := \{P^{crit}_{2n}\}_{n=0}^\infty$ and $\Pi_{\text{odd}}:=\{P^{crit}_{2n+1}\}_{n=0}^\infty$ form two disjoint nested puzzle sequences for $F_\nu$ at the critical point $1$.

\begin{figure}[h]
\centering
\includegraphics[scale=0.6]{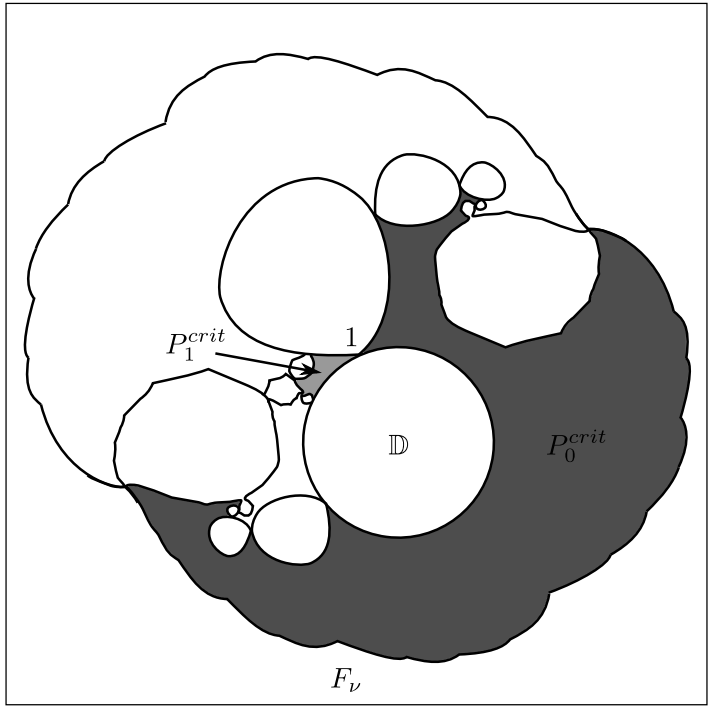}
\caption{The $0$th and $1$st critical puzzle piece for $F_\nu$.}
\label{fig:critpuzzles}
\end{figure}

\begin{lem} \label{disjointness}
Let $\mathcal{A}_\infty \cup F_\nu(\mathcal{A}_\infty)$ be the immediate attracting basin of the superattracting 2-periodic orbit $\{\infty, F_\nu(\infty)\}$ for $F_\nu$. Then there exists $N \in \mathbb{N}$ such that for all $n \geq N$, the $n$th critical puzzle piece $P^{crit}_n$ is disjoint from the closure of $\mathcal{A}_\infty \cup F_\nu(\mathcal{A}_\infty)$.
\end{lem}

\begin{proof}
The result follows immediately from proposition \ref{disjointness for rnu}.
\end{proof}

\begin{thm} \label{yampolthm}
For all $n$ sufficiently larger than the constant $N$ in lemma \ref{disjointness}, we have the following inequality:
\begin{displaymath}
\frac{\emph{diam}(P^{crit}_n)}{\emph{diam}(A_{-n})} \leq C_1 \sqrt[3]{\frac{\emph{diam}(P^{crit}_{n-1})}{\emph{diam}(A_{-(n-1)})}} + C_2,
\end{displaymath}
where $C_1$ and $C_2$ are universal constants.
\end{thm}

\begin{proof}
Similarly to [YZ], we first lift a suitable inverse branch of $F_\nu$ to the universal covering space.

Define the exponential map $\text{Exp}: \mathbb{C} \to \mathbb{C}$ by
\begin{displaymath}
\text{Exp}(z) := e^{2\pi i z}.
\end{displaymath}
Let $I = (\tau - 1, \tau) \subset \mathbb{R}$ be an open interval such that $0 \in I$, and
\begin{displaymath}
\text{Exp}(\tau) = \text{Exp}(\tau - 1) = F_\nu(1).
\end{displaymath}
Let
\begin{displaymath}
\text{Log} : S^1 \setminus \{F_\nu(1)\} \to I
\end{displaymath}
be the inverse of Exp restricted to $I$. The $n$th critical interval is defined as
\begin{displaymath}
I_n := \text{Log}(A_n).
\end{displaymath}
Denote the component of $\text{Exp}^{-1}(P^{crit}_n)$ intersecting $I$ by $\hat{P}^{crit}_n$. 

Define
\begin{displaymath}
\mathcal{A} := \overline{\mathcal{A}_0 \cup F_\nu(\mathcal{A}_0) \cup \mathcal{A}_\infty \cup F_\nu(\mathcal{A}_\infty)},
\end{displaymath}
and let $S \subset \mathbb{C}$ be the universal covering space of $\hat{\mathbb{C}} \setminus \mathcal{A}$ with the covering map $\text{Exp}|_S : S \to \hat{\mathbb{C}} \setminus \mathcal{A}$. For any given interval $J \subset \mathbb{R}$, we denote
\begin{displaymath}
S_J := (S \setminus \mathbb{R}) \cup J.
\end{displaymath}

The restriction of the map $F_\nu$ to $S^1$ is a homeomorphism, and hence, has an inverse. We define a lift $\phi : I \to I$ of $(F_\nu|_{\partial{\mathbb{D}}})^{-1}$ by
\begin{displaymath}
\phi(x) := \text{Log} \circ F_\nu^{-1} \circ \text{Exp}(x).
\end{displaymath}
Note that $\phi$ is discontinuous at $\text{Log}(F_\nu^2(1))$, which is mapped to $\tau-1$ and $\tau$ by $\phi$. Let $n \in \mathbb{N}$. By the combinatorics of critical circle maps, the $k$th iterate of $\phi$ on $I_n$ is continuous for all $1 \leq k \leq q_n$. By monodromy theorem, $\phi^k$ extends to a conformal map on $S_{I_n}$.

For $z \in S_J$, let $l_z$ and $r_z$ be the line segment connecting $z$ to $\tau-1$ and $z$ to $\tau$ respectively. The smaller of the outer angles formed between $l_z$ and $(-\infty, \tau-1)$, and $r_z$ and $(\tau, +\infty)$ is denoted $\widehat{(z, J)}$. 

\begin{figure}[h]
\centering
\includegraphics[scale=0.6]{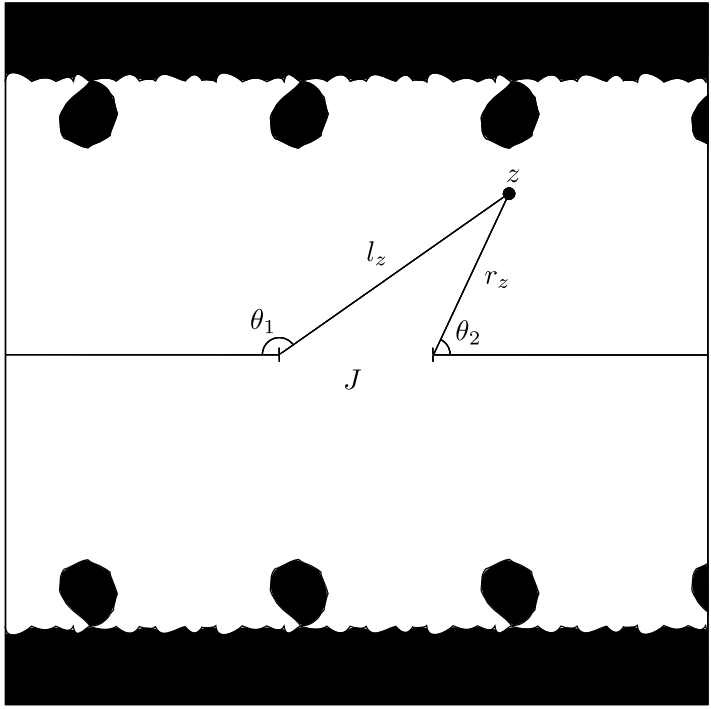}
\caption{Illustration of $\widehat{(z, J)} = \text{min}(\theta_1, \theta_2)$.}
\label{fig:Sangle}
\end{figure}

Denote the hyperbolic distance in $S_J$ by $\text{dist}_{S_J}$. A hyperbolic neighbourhood $\{z \in S_J \hspace{2mm} | \hspace{2mm} \text{dist}_{S_J}(z, J)\}$ of $J$ forms an angle $\theta \in (0, \pi)$ with $\mathbb{R}$. Denote this neighbourhood by $G_\theta(J)$. Observe that $G_\theta(J) \subset \{z \in S_J \hspace{2mm} | \hspace{2mm} \widehat{(z, J)} > \theta\}$.

For $n \in \mathbb{N}$, define $E_n \subset S^1$ as the open arc containing $1$ with end points at $F_\nu^{q_{n+1}}(1)$, and $F_\nu^{q_n - q_{n+1}}(1)$. Observe that $E_n$ contains the critical arcs $A_n$ and $A_{n+1}$. Define
\begin{displaymath}
G^n_\theta := G_\theta(\text{Log}(E_n)).
\end{displaymath}

\begin{figure}[h]
\centering
\includegraphics[scale=0.6]{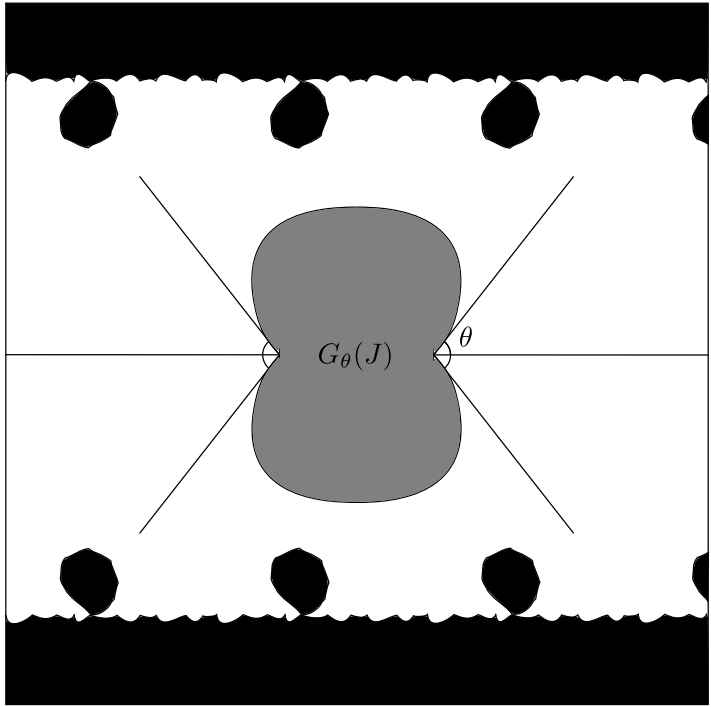}
\caption{Illustration of the hyperbolic neighbourhood $G_\theta(J)$.}
\label{fig:Sneighbourhood}
\end{figure}

Consider the constant $N$ in lemma \ref{disjointness}. Since $P^{crit}_N \cup P^{crit}_{N+1}$ is disjoint from the closure of $\mathcal{A}$, it is contained in some annulus $E \Subset \hat{\mathbb{C}} \setminus \mathcal{A}$. Let $\breve{S} \Subset S$ be the universal cover of $E$ with the covering map $\text{Exp}|_{\breve{S}}$. Choose $\theta$ such that $\hat{P}^{crit}_{N+2} \cup \hat{P}^{crit}_{N+3} \subset G^{N+1}_\theta$. Then we have $\hat{P}^{crit}_n \subset G^{N+1}_\theta$ for all $n \geq N+3$.

Now, suppose we are given $n \geq N+3$. Let
\begin{equation} \label{intorbit}
J_0 := I_n, J_{-1} := \phi(J_0), \ldots, J_{-q_n} := \phi^{q_n}(I_n),
\end{equation}
be the orbit of $I_n$ under $\phi$. Given any point $z_0 \in S_{J_0}$, let
\begin{equation} \label{ptorbit}
z_0, z_{-1} := \phi(z_0), \ldots, z_{-q_n} := \phi^{q_n}(z_0),
\end{equation}
be the orbit of $z_0$ under $\phi$.

The following three lemmas are adaptations of lemma 2.1, 4.2 and 4.4 in \cite{Y1} and lemma 6.1, 6.2 and 6.3 in \cite{YZ}:

\begin{lem} \label{lem6.1}
Consider the orbit \eqref{ptorbit}. Let $k \leq q_n-1$. Assume that for some $i$ between $0$ and $k$, $z_i \in \breve{S}$, and $\widehat{(z_{-i}, J_{-i})} > \epsilon$. Then we have
\begin{displaymath}
\frac{\text{dist}(z_{-k}, J_{-k})}{|J_{-k}|} \leq C \frac{\text{dist}(z_{-i}, J_{-i})}{|J_{-i}|}
\end{displaymath}
for some constant $C = C(\epsilon, \breve{S})>0$.
\end{lem}

\begin{lem} \label{lem6.2}
Let $J$ and $J'$ be two consecutive returns of the orbit \eqref{intorbit} of $J_0$ to $I_m$ for $1 < m < n$, and let $\zeta$ and $\zeta'$ be the corresponding points of the inverse orbit \eqref{ptorbit}. If $\zeta \in G^m_\theta$, then either $\zeta' \in G^m_\theta$ or $\widehat{(\zeta', J')}>\epsilon$ and $\text{dist}(\zeta', J')< C|I_m|$, where the constants $\epsilon$ and $C$ are independent of $m$.
\end{lem}

\begin{lem} \label{lem6.3}
Let $J$ be the last return of the orbit \eqref{intorbit} to the interval $I_m$ preceding the first return to $I_{m+1}$ for $1 \leq m \leq n-1$, and let $J'$ and $J''$ be the first two returns to $I_{m+1}$. Let $\zeta$, $\zeta'$ and $\zeta''$ be the corresponding points in the inverse orbit \eqref{ptorbit}, so that $\zeta' = \phi^{q_m}(\zeta)$ and $\zeta'' = \phi^{q_{m+2}}(\zeta')$. Suppose that $\zeta \in G^m_\theta$. Then either $\widehat{(\zeta'', I_{m+1})}>\epsilon$ and $\text{dist}(\zeta'', J'') < C |I_{m+1}|$, or $\zeta'' \in G^{m+1}_\theta$, where the constants $\epsilon$ and $C$ are independent of $m$. 
\end{lem}

The interested reader can follow the proofs of lemma \ref{lem6.1}, \ref{lem6.2} and \ref{lem6.3}, and the rest of the proof of theorem \ref{yampolthm} in \cite{YZ} \emph{mutatis mutandis}.
\end{proof}

\begin{cor} \label{crit shrink}
For all $n$ sufficiently larger than the constant $N$ in lemma \ref{disjointness}, $\emph{diam}(P^{crit}_n)$ is $K$-commensurate to $\emph{diam}(A_{-n})$ for some universal constant $K>1$. Consequently, $\emph{diam}(P^{crit}_n) \to 0$ as $n \to \infty$.
\end{cor}

\begin{proof}
It suffices to show that any sequence of positive numbers $\{a_n\}_{n=0}^\infty$ satisfying the relation
\begin{displaymath}
a_n \leq C_1 \sqrt[3]{a_{n-1}} + C_2
\end{displaymath}
for all $n$ is bounded.

Consider the sequence $\{b_n\}_{n=0}^\infty$ defined inductively by
\begin{enumerate} [label = \roman*)]
\item $b_0 =$ max$(1, a_0)$,
\item $b_n = C \sqrt[3]{b_{n-1}}$,
\end{enumerate}
where $C$ is chosen so that
\begin{displaymath}
C \sqrt[3]{k} \geq C_1 \sqrt[3]{k} + C_2
\end{displaymath}
for all $k \geq 1$. It is easy to see that $b_n \geq a_n$ for all $n$.

A straightforward computation shows that
\begin{displaymath}
b_n = C^{1 + \frac{1}{3} + \ldots + \frac{1}{3^{n-1}}} \sqrt[3^{n-1}]{b_0} \xrightarrow{n \to \infty} C^{\frac{3}{2}}.
\end{displaymath}
Hence, $\{b_n\}_{n=0}^\infty$ and therefore, $\{a_n\}_{n=0}^\infty$ are bounded.
\end{proof}

The following result we record for later use:

\begin{lem} \label{comm disc}
For all $n$ sufficiently large, the $n$th critical puzzle piece $P^{crit}_n$ contains a Euclidean disc $D_n$ such that \emph{diam}$(D_n)$ is $K$-commensurate to \emph{diam}$(P^{crit}_n)$ for some universal constant $K>1$.
\end{lem}

\begin{proof}
Let $D_1$ be a disc centered at $1$ such that $F_\nu^{q_n}(1) \in \partial D_1$. The map $F_\nu^{q_n}|_{A_n}$ has a well defined inverse branch which extends to $D_1$. Denote this inverse branch by $\psi_n$. As a consequence of real a priori bounds, we have the following estimate:
\begin{displaymath}
\frac{1}{|K_1|} \leq |\psi_n'(1)| \leq |K_1|,
\end{displaymath}
where $K_1$ is some universal constant independent of $n$.

Observe that the preimage of $\mathbb{D}$ under $F_\nu$ consists of two connected components $U_{\text{in}} \subset \mathbb{D}$ and $U_{\text{out}} \subset \mathbb{C} \setminus \overline{\mathbb{D}}$. Moreover, $\overline{U_{\text{in}}} \cap \overline{U_{\text{out}}} = \{1\}$. It is not difficult to see that $\psi_n$ extends to $U_{\text{out}}$, and that $\psi_n(U_{\text{out}}) \subset P^{crit}_n$.

Now, choose a subdisc $D_2 \subset D_1 \cap U_{\text{out}}$ such that the annulus $A = D_1 \setminus \overline{D_2}$ satisfies the following estimate
\begin{displaymath}
\frac{1}{|K_2|} \leq \text{mod}(A) \leq |K_2|,
\end{displaymath}
for some universal constant $K_2$ independent of $n$. By Koebe distortion theorem, $\psi_n$ has uniformly bounded distortion on $D_2$. Since $\psi_n(D_2) \subset \psi_N(U_\text{out}) \subset P^{crit}_n$, the result follows.
\end{proof}

\newpage

\section{The Proof of the Shrinking Theorem} \label{shrink proofs}

We are ready to prove the shrinking theorem stated at the end of section \ref{puzzle part}. The proof will be split into three propositions.

\begin{prop} \label{alpha shrinkage}
If $\Pi_t$ is a nested puzzle sequence such that $L(\Pi_t)$ contains $\beta$ or $\kappa$, then $\Pi_t$ shrink to a point.
\end{prop}

\begin{proof}
We prove the result in the case where $L(\Pi_t)$ contains $\kappa$. The proof of the other case is similar.

Since $L(\Pi_t)$ contains $\kappa$, it follows that $t = 0$. Observe that $L(\Pi_0)$ is invariant under $R_\nu$. Hence, $L(\Pi_0) \cap \partial \mathcal{S}_0 = \varnothing$.

Let $D_r$ be a disc of radius $r>0$ centered at $\kappa$. Since $\kappa$ is a repelling fixed point, if $r$ is sufficiently small, then $D_r$ is mapped into itself by an appropriate inverse branch of $R_\nu$. This inverse branch extends to a map $g : N \to N$, where $N$ is a neighbourhood of $L(\Pi_0)$ which is disjoint from $\partial \mathcal{S}_0$, and therefore the closure of the post critical set for $R_\nu$.

Any set compactly contained within $N$ converges to $\kappa$ under iteration of $R_\nu$. It follows that $L(\Pi_0) = \{\kappa\}$.
\end{proof}

For the proof of the remaining two propositions, it will be more convenient for us to work with the Blaschke product $F_\nu$ rather than $R_\nu$ itself. It is clear from the definition that a nested puzzle sequence for $R_\nu$ shrinks if and only if the corresponding nested puzzle sequence for $F_\nu$ shrinks.

\begin{prop} \label{siegel shrinkage}
If $\Pi_t$ is a nested puzzle sequence such that $1 \in L(\Pi_t)$, then $\Pi_t$ shrink to $1$
\end{prop}

\begin{proof}
Recall the definition of critical puzzle pieces $\{P^{crit}_n\}_{n=0}^\infty$ for $F_\nu$ in section \ref{secapriori}. Let $\hat{\Pi}_{\text{even}}$ and $\hat{\Pi}_{\text{odd}}$ be the maximal nested puzzle sequence containing $\{P^{crit}_{2n}\}_{n=0}^\infty$ and $\{P^{crit}_{2n+1}\}_{n=0}^\infty$ respectively. Corollary \ref{crit shrink} and proposition \ref{maximal shrink} imply that $\hat{\Pi}_{\text{even}}$ and $\hat{\Pi}_{\text{odd}}$ both shrink to $1$. By proposition \ref{num sequence}, there is no other maximal nested puzzle sequence at $1$.
\end{proof}

For the proof of the final proposition, we need the following lemma.

\begin{lem} \label{shrinking lemma}
Let $f : \hat{\mathbb{C}} \to \mathbb{C}$ be a rational map of degree $d > 1$. Let $\{(f|_{U})^{-n}\}_{n=0}^\infty$ be a family of univalent inverse branches of $f$ restricted to a domain $U$. Suppose $U \cap J(f) \neq \varnothing$. If $V \Subset U$, then
\begin{displaymath}
\emph{diam}((f|_U)^{-n}(V)) \to 0
\end{displaymath}
as $n \to \infty$.
\end{lem}

\begin{prop} \label{candidate shrinkage}
Let $z_0$ be a point in the Julia set $J(R_\nu)$ which is not an iterated preimage of $\kappa$, $\beta$ or $1$. If $\Pi_t$ is a nested puzzle sequence such that $z_0 \in L(\Pi_t)$, then $\Pi_t$ shrinks to $z_0$.
\end{prop}

\begin{proof}
Let
\begin{displaymath}
\mathcal{O} = \{z_n\}_{n=0}^\infty
\end{displaymath}
be the forward orbit of $z_0$ under $F_\nu$. The proof splits into two cases.

\vspace{2.5mm}

\noindent \emph{Case 1.} Suppose there exists some critical puzzle piece $P^{crit}_M$ such that
\begin{displaymath}
\mathcal{O} \cap P^{crit}_M = \varnothing.
\end{displaymath}
Let $z_\infty$ be an accumulation point of $\mathcal{O}$, and let $P^\infty$ be the puzzle piece of level $M$ containing $z_\infty$. Observe that the orbit of the critical point $1$ is dense in $\partial \mathbb{D}$. Hence, $P^\infty$ must be disjoint from $\partial \mathbb{D}$, since otherwise, $P^\infty$ would map into $P^{crit}_M$ by some appropriate inverse branch of $F_\nu$.

Let $U \subset \mathbb{C} \setminus \overline{\mathbb{D}}$ be a neighbourhood of $P^\infty$, and choose a subsequence of orbit points $\{z_{n_k}\}_{k=0}^\infty$ from $\mathcal{O}$ such that $z_{n_k} \in P^\infty$. For each $k$, let
\begin{displaymath}
g_k : U \to \mathbb{C}
\end{displaymath}
be the inverse branch of $F_\nu^{n_k}$ that maps $z_{n_k}$ to $z_0$. Since $P^\infty$ intersects the Julia set for $F_\nu$, the nested puzzle sequence
\begin{displaymath}
\Pi := \{g_k(P^\infty)\}_{k=0}^\infty
\end{displaymath}
must shrink to $z_0$ by lemma \ref{shrinking lemma}.

\vspace{2.5mm}

\noindent \emph{Case 2.} Suppose the critical point $1$ is an accumulation point of $\mathcal{O}$. Then there exists an increasing sequence of numbers $\{n_k\}_{k=0}^\infty$ such that
\begin{displaymath}
\mathcal{O} \cap P^{crit}_{n_k} \neq \varnothing.
\end{displaymath}
Fix $k$, and let $z_{m_k}$ be the first orbit point that enters the critical puzzle piece $P^{crit}_{n_k}$. Let
\begin{equation}
P^{-n} \subset F_\nu^{-n}(P^{crit}_{n_k})
\end{equation}
be the $n$th pull back of $P^{crit}_{n_k}$ along the orbit
\begin{equation} \label{orbit}
z_0 \mapsto z_1 \mapsto \ldots \mapsto z_{m_k}.
\end{equation}

Suppose that $P^{-n}$ intersects $1$ for some $n>0$. Then for all $m \leq n$, $P^{-m}$ must intersect $\partial \mathbb{D}$. Recall that $P^{crit}_{n_k}$ contains the the preimage arc $A_{-n_k}$. Hence, for every $m \leq n$, $P^{-m}$ contains the $m$th preimage of $A_{-n_k}$ under $F_\nu|_{\partial \mathbb{D}}$. By the combinatorics of critical circle maps, it follows that $P^{-q_{n_k}}$ must be the first puzzle piece in the backward orbit $\{P^{-1}, P^{-2}, \ldots, P^{-m_k}\}$ to intersect $1$.

Since there are exactly two maximal nested puzzle sequences whose limit contains $1$, all puzzle pieces of level $n > n_k + q_{n_k}$ which intersect $1$ must be contained in either $P^{crit}_{n_k}$ or $P^{-q_{n_k}}$. Either case would contradict the fact that $z_{m_k}$ is the first orbit point to enter $P^{crit}_{n_k}$. Therefore, $P^{-n}$ does not intersect $1$ for all $n \geq q_{n_k}$.

Let $m \leq m_k$ be the last moment when the backward orbit of $P^0 = P^{crit}_{n_k}$ intersect $\partial \mathbb{D}$. By theorem \ref{real bound}, corollary \ref{crit shrink} and combinatorics of critical circle maps, the distance between $P^{-m}$ and $F_\nu(1)$ is commensurate to diam$(P^{-m})$. Hence, the distance between $P^{-m-1}$ and $1$ is commensurate to diam$(P^{-m-1})$. Therefore, by theorem \ref{real bound} and Koebe distortion theorem, the inverse branch of $F_\nu^{m_k}$ along the orbit \eqref{orbit} can be expressed as either
\begin{displaymath}
F_\nu^{-m_k}|_{P^{crit}_{n_k}} = \eta
\end{displaymath}
if $1 \notin P_n$ for all $n > 0$, or
\begin{displaymath}
F_\nu^{-m_k}|_{P^{crit}_{n_k}} = \zeta_1 \circ Q \circ \zeta_2
\end{displaymath}
if $1 \in P^{-q_{n_k}}$, where $\eta$, $\zeta_1$ and $\zeta_2$ are conformal maps with bounded distortion, and $Q$ is a branch of the cubic root.

Now, by lemma \ref{comm disc}, $P^{crit}_{n_k}$ contains a Euclidean disc $D_{n_k}$ such that diam$(D_{n_k})$ is commensurate to diam$(P^{crit}_{n_k})$. The above argument implies that the puzzle piece $P^{-m_k}$ must also contain a Euclidean disc $D$ such that diam$(D)$ is commensurate to  diam$(P^{-m_k})$. Hence, $\text{diam}(P^{-m_k}) \to 0$ as $k \to \infty$, and the nested puzzle sequence
\begin{displaymath}
\Pi := \{P^{-m_k}\}_{k=0}^\infty
\end{displaymath}
must shrink to $z_0$.
\end{proof}

As an application of the shrinking theorem, we prove that every infinite bubble ray for $R_\nu$ lands.

\begin{prop} \label{everylanding}
Every infinite bubble ray for $R_\nu$ lands.
\end{prop}

\begin{proof}
Let $\mathcal{R}_t$ be an infinite bubble ray, and let $\Omega$ be its accumulation set. If $t$ is a dyadic rational, then $\mathcal{R}_t$ lands on an iterated preimage of $\kappa$. Otherwise, there exists a unique nested maximal puzzle sequence $\Pi_t = \{P_{[s_k, t_k]}\}_{k=0}^\infty$ with external angle equal to $t$. By corollary \ref{ray in puzzle}, $\Omega$ must be contained in $P_{[s_k, t_k]}$ for all $k \in \mathbb{N}$. The result now follows from the shrinking theorem.
\end{proof}

\newpage

\section{The Proof of Conformal Mateability (Theorem B)} \label{final sec}

We are ready to prove that $R_\nu$ is a conformal mating of $f_\mathbf{B}$ and $f_\mathbf{S}$. Recall the maps $\Phi_\mathbf{B}$ and $\Phi_\mathbf{S}$ in theorem \ref{interior maps} defined on the union of the closure of every bubble in $\mathbf{B}$ and $\mathbf{S}$ respectively. Our first task is to continuously extend $\Phi_\mathbf{B}$ and $\Phi_\mathbf{S}$ to the filled Julia sets $K_\mathbf{B} = \overline{\mathbf{B}}$ and $K_\mathbf{S} = \overline{\mathbf{S}}$. For brevity, we will limit our discussion to $\Phi_\mathbf{S}$. The map $\Phi_\mathbf{B}$ can be extended in a completely analogous way.

Let $\tilde{\Phi}_\mathbf{S} : J_\mathbf{S} \to J(R_\nu)$ be the map defined as follows. For $x \in J_\mathbf{S}$, let $\Pi^\mathbf{S}_t = \{P^\mathbf{S}_{[s_k, t_k]}\}_{k=0}^\infty$ be a maximal nested puzzle sequence whose limit contains $x$. By the shrinking theorem, the corresponding maximal nested puzzle sequence $\Pi_t = \{P_{[s_k, t_k]}\}_{k=0}^\infty$ for $R_\nu$ must shrink to a single point, say $y \in J(R_\nu)$. Define $\tilde{\Phi}_\mathbf{S}(x) := y$. We claim that $\tilde{\Phi}_\mathbf{S}$ is a continuous extension of $\Phi_\mathbf{S}$ on $J_\mathbf{S}$.

\begin{prop} \label{cont ext}
Let $S \subset \mathbf{S}$ be a bubble. If $x \in \partial S$, then $\tilde{\Phi}_\mathbf{S}(x) = \Phi_\mathbf{S}(x)$.
\end{prop}

\begin{proof}
Let $z := \Phi_\mathbf{S}(x)$. By the definition of puzzle partitions, $z \in P_{[s_k, t_k]}$ for all $k \geq 0$. The result follows.
\end{proof}

\begin{prop} \label{well defined}
The map $\tilde{\Phi}_\mathbf{S} : J_\mathbf{S} \to J(R_\nu)$ is well defined.
\end{prop}

\begin{proof}
Suppose there are two maximal nested puzzle sequences at $x \in J_\mathbf{S}$. By proposition \ref{num seq siegel}, $x$ is either an iterated preimage of $\mathbf{k}_\mathbf{S}$ or $0$. The first case follows from corollary \ref{ray in puzzle}. The second case follows from proposition \ref{cont ext}.
\end{proof}

\begin{prop}
Define $\Phi_\mathbf{S}(x) := \tilde{\Phi}_\mathbf{S}(x)$ for all $x \in J_\mathbf{S}$. The extended map $\Phi_\mathbf{S} : K_\mathbf{S} \to \hat{\mathbb{C}}$ is continuous.
\end{prop}

\begin{proof}
It suffices to show that if $\{x_i\}_{i=0}^\infty \subset K_\mathbf{S}$ is a sequence converging to $x \in J_\mathbf{S}$, then the sequence of image points $\{y_i = \Phi_\mathbf{S}(x_i)\}_{i=0}^\infty$ converges to $y = \Phi_\mathbf{S}(x)$. The proof splits into four cases:
\begin{enumerate}[label=\roman*)]
\item $x$ is an iterated preimage of $0$.
\item There exists a unique bubble $S \subset \mathbf{S}$ such that $x \in \partial S$.
\item $x$ is an iterated preimage of $\mathbf{k}_\mathbf{S}$.
\item Otherwise.
\end{enumerate}
\vspace{2.5mm}

\noindent \emph{Case i)} By proposition \ref{siegel preimage 0}, there exist exactly two bubbles $S_1$ and $S_2$ which contain $x$ in their boundary. Moreover, we have $\{x\} = \overline{S_1} \cap \overline{S_2}$. By proposition \ref{cont ext}, any subsequence of $\{x_i\}_{i=0}^\infty$ contained in $\overline{S_1} \cup \overline{S_2}$ is mapped under $\Phi_\mathbf{S}$ to a sequence which converges to $y$. Hence, we may assume that $x_i$ is not contained $\overline{S_1} \cup \overline{S_2}$ for all $i \geq 0$. 

By proposition \ref{num seq siegel}, there are exactly two maximal nested puzzle sequences $\Pi^\mathbf{S}_t = \{P^\mathbf{S}_{[s_k, t_k]}\}_{k=0}^\infty$ and $\Pi^\mathbf{S}_v = \{P^\mathbf{S}_{[u_k, v_k]}\}_{k=0}^\infty$ whose limit contains $x$. Let $D_r(x)$ be a disc of radius $r > 0$ centered at $x$. For every $k$, we can choose $r_k >0$ sufficiently small such that $D_{r_k}(x) \cap \mathcal{P}^\mathbf{S}_k = D_{r_k}(x) \cap (\overline{S_1} \cup \overline{S_2})$. Let $N_k \geq 0$ be large enough such that $\{x_i\}_{i=N_k}^\infty$ is contained in $D_{r_k}(x)$. This implies that $\{x_i\}_{i=N_k}^\infty \subset P^\mathbf{S}_{[s_k, t_k]} \cup P^\mathbf{S}_{[u_k, v_k]}$. It is easy to see that the sequence of image points $\{y_i = \Phi_\mathbf{S}(x_i)\}_{i=N_k}^\infty$ must be contained $P_{[s_k, t_k]} \cup P_{[u_k, v_k]}$. By proposition \ref{well defined}, $\Pi_t = \{P_{[s_k, t_k]}\}_{k=0}^\infty$ and $\Pi_v = \{P_{[u_k, v_k]}\}_{k=0}^\infty$ both converge to $y$, and the result follows.

\vspace{2.5mm}

\noindent \emph{Case ii)} The proof is very similar to Case i), and hence, it will be omitted.

\vspace{2.5mm}

\noindent \emph{Case iii)} Since $x$ is an iterated preimage of $\mathbf{k}_\mathbf{S}$, it must be the landing point of some bubble ray $\mathcal{R}^\mathbf{S}_t$, where $t \in \mathbb{R} /\mathbb{Z}$ is a dyadic rational. By corollary \ref{ray in puzzle}, $y$ is the landing point of the corresponding bubble ray $\mathcal{R}^\mathcal{S}_t$. Any subsequence of $\{x_i\}_{i=0}^\infty$ contained in $\mathcal{R}^\mathbf{S}_t$ is mapped under $\Phi_\mathbf{S}$ to a sequence in $\mathcal{R}^\mathcal{S}_t$ which converges to $y$. Hence, we may assume that $x_i$ is not contained $\mathcal{R}^\mathbf{S}_t$ for all $i \geq 0$.

The remainder of the proof is very similar to Case i), and hence, it will be omitted.

\vspace{2.5mm}

\noindent \emph{Case iv)} By proposition \ref{num seq siegel}, there exists a unique maximal nested puzzle sequences $\Pi^\mathbf{S}_t = \{P^\mathbf{S}_{[s_k, t_k]}\}_{k=0}^\infty$ whose limit contains $x$. Let $D_r(x)$ be a disc of radius $r > 0$ centered at $x$. Since $x$ is not contained the puzzle partition $\mathcal{P}^\mathbf{S}_n$ of any level $n \geq 0$, it follows that for every $k \geq 0$, there exists $r_k >0$ sufficiently small such that $D_r(x) \subset P^\mathbf{S}_{[s_k, t_k]}$. Thus, there exists $N_k \geq 0$ such that $\{x_i\}_{i=N_k}^\infty$ is contained in $P^\mathbf{S}_{[s_k, t_k]}$. It is easy to see that the sequence of image points $\{y_i = \Phi_\mathbf{S}(x_i)\}_{i=N_k}^\infty$ must be contained in the corresponding puzzle piece $P_{[s_k, t_k]}$ for $R_\nu$. Since the nested puzzle sequence $\Pi_t = \{P_{[s_k, t_k]}\}_{k=0}^\infty$ must shrink to $y$, the result follows.
\end{proof}

\begin{prop} \label{ext angle match}
Let $t \in \mathbb{R} /\mathbb{Z}$, and let $x \in J_\mathbf{B}$ and $y \in J_\mathbf{S}$ be the landing point of the external ray for $f_\mathbf{B}$ and $f_\mathbf{S}$ with external angle $-t$ and $t$ respectively. Then $\Phi_\mathbf{B}(x) = \Phi_\mathbf{S}(y)$.
\end{prop}

\begin{proof}
Consider the nested puzzle sequences $\Pi^\mathbf{B}_{t} = \{P^\mathbf{B}_{[s_k, t_k]}\}_{k=0}^\infty$, $\Pi^\mathbf{S}_t = \{P^\mathbf{S}_{[s_k, t_k]}\}_{k=0}^\infty$ and $\Pi_t = \{P_{[s_k, t_k]}\}_{k=0}^\infty$. By proposition \ref{shrink to ray} and \ref{shrink to ray 2}, $\Pi^\mathbf{B}_{t}$ and $\Pi^\mathbf{S}_t$ shrink to $x$ and $y$ respectively. Let $z$ be the point that $\Pi_t$ shrinks to. By definition, $\Phi_\mathbf{B}(x) = z = \Phi_\mathbf{S}(y)$.
\end{proof}

\begin{proof} [Proof of theorem B] \hspace{1mm}

\vspace{3mm}

We verify the conditions in proposition \ref{mating def}. Let $f_{c_1} = f_\mathbf{B}$, $f_{c_2} = f_\mathbf{S}$, $\Lambda_1 = \Phi_\mathbf{B}$, $\Lambda_2 = \Phi_\mathbf{S}$, and $R = R_\nu$. Clearly, conditions (ii) and (iii) are satisfied. It remains to check condition (i).

Let $\tau_\mathbf{B} : \mathbb{R} / \mathbb{Z} \to J_\mathbf{B}$ and $\tau_\mathbf{S} : \mathbb{R} / \mathbb{Z} \to J_\mathbf{S}$ be the Carath\'eodory loop for $f_\mathbf{B}$ and $f_\mathbf{S}$ respectively (refer to section \ref{defnmating} for the definition of Carath\'eodory loop). Define $\sigma_\mathbf{B}(t) := \tau_\mathbf{B}(-t)$. By proposition \ref{ext angle match}, the following diagram commutes:
\comdia{\mathbb{R}/\mathbb{Z}}{J_\mathbf{B}}{J_\mathbf{S}}{J(R_\nu)}{\sigma_\mathbf{B}}{\Phi_\mathbf{S}}{\tau_\mathbf{S}}{\Phi_\mathbf{B}}
It follows that if $z \sim_{ray} w$, then $z$ and $w$ are mapped to the same point under $\Phi_\mathbf{B}$ or $\Phi_\mathbf{S}$.

To check the converse, it suffices to prove that for $z, w \in J_\mathbf{S}$, if $\Phi_\mathbf{S}(z) = \Phi_\mathbf{S}(w) = x \in J(R_\nu)$, then $z \sim_{ray} w$. First, observe that $\Phi_\mathbf{S}$ maps iterated preimages of $0$ homeomorphically onto the iterated preimages of $1$. Similarly, $\Phi_\mathbf{S}$ maps iterated preimages of $\mathbf{k}_\mathbf{S}$ homeomorphically onto the iterated preimages of $\kappa$. Now, by proposition \ref{num sequence}, two distinct maximal nested sequences for $R_\nu$ shrink to $x$ if and only if $x$ is an iterated preimage of $1$, $\kappa$ or $\beta$. If $x$ is an iterated preimage of $1$ or $\kappa$, then $z$ must be equal to $w$. If $x$ is an iterated preimage of $\beta$, then $z \sim_{ray} w$.
\end{proof}

\end{document}